\documentclass[11pt]{article}
\usepackage{amssymb, amsmath,amsthm}
\usepackage{mathrsfs, yfonts, color, calligra}
\usepackage[all]{xy}
\usepackage{enumitem}
\setitemize{itemsep=1pt,topsep=0pt,parsep=0pt,partopsep=0pt}
 \usepackage{mathptm}
\usepackage{multirow,float,longtable}
\usepackage{pgffor}
\usepackage{tikz}
\usepackage{changes}

 \usetikzlibrary{calc}

\newtheorem{Def}{Definition}[section]
\newtheorem{Prop}[Def]{Proposition}
\newtheorem{Theo}[Def]{Theorem}
\newtheorem{Lem}[Def]{Lemma}
\newtheorem{Koro}[Def]{Corollary}

 \DeclareMathOperator{\add}{add}

\DeclareMathOperator{\domdim}{dom.dim}
 \DeclareMathOperator{\End}{End}
 
 \DeclareMathOperator{\Ext}{Ext}
 
 \DeclareMathOperator{\gldim}{gl.dim}

 \DeclareMathOperator{\repdim}{rep.dim}

 \DeclareMathOperator{\stHom}{\underline{Hom}}

 \DeclareMathOperator{\rd}{rd}
 \DeclareMathOperator{\SE}{{\sf SE}}
 \DeclareMathOperator{\rigdim}{rig.dim}
 
 \DeclareMathOperator{\sgn}{sgn}

\newcommand{\defCategory}[2]{
  \newcommand{#1}{#2\defvariable}}

\newcommand{\defvariable}[2][]{
\if\relax\detokenize{#1}\relax  %if the first arg is empty
\if\relax\detokenize{#2}\relax
    \else  ({#2})  \fi
    \else  ^{{\rm #1}}({#2})  \fi}

\defCategory{\C}{\mathscr{C}}
\defCategory{\K}{\mathscr{K}}
\defCategory{\D}{\mathscr{D}}
\defCategory{\wseq}{{\bf k}}
\newcommand{\kk}{\wseq{}}

\foreach \x in {A,...,Z}{%
\expandafter\xdef\csname scr\x\endcsname{\noexpand\ensuremath{\noexpand\mathscr{\x}}}
\expandafter\xdef\csname bb\x\endcsname{\noexpand\ensuremath{\noexpand\mathbb{\x}}}}

\def\modcat#1{{#1}\mbox{-}{\sf mod}}

\def\stmodcat#1{{#1}\mbox{-}\underline{\sf mod}}

\newcommand{\lra}{\longrightarrow}

\newcommand{\ra}{\rightarrow}

\newcommand{\numseq}[1]{{\bf #1}}
\newcommand{\Fb}{{\sf F}}

\renewcommand{\leq}{\leqslant}
\renewcommand{\geq}{\geqslant}

\newcommand{\rem}[1]{[#1]}

\title{Rigidity degrees of indecomposable modules over representation-finite self-injective algebras}
\author{Wei Hu, Xiaojuan Yin$^*$}
\date{}

\begin{document}
\maketitle

\renewcommand{\thefootnote}{\alph{footnote}}
\setcounter{footnote}{-1} \footnote{ $^*$ Corresponding author.}
%Email: xicc@cnu.edu.cn; Fax: 0086 10 68903637.}
\renewcommand{\thefootnote}{\alph{footnote}}
\setcounter{footnote}{-1} 
\renewcommand{\thefootnote}{\alph{footnote}}
\setcounter{footnote}{-1} 

\begin{abstract}
The rigidity degree of a generator-cogenerator determines the dominant dimension of its endomorphism algebra, and is closely related to a recently introduced homological dimension --- rigidity dimension. In this paper, we give explicit formulae for the rigidity degrees of all indecomposable modules over representation-finite self-injective algebras by developing combinatorial methods from the Euclidean algorithm.  As an application, the  rigidity dimensions of some algebras of types $A$ and $E$ are given.
\end{abstract}

%\tableofcontents

\section{Introduction}
Dominant dimension and global dimension are two fundamental homological dimensions of finite dimensional algebras. Their interplay occurs in Auslander's definition of representation dimensions \cite{Auslander1970}, or more generally, Iyama's definition of higher representation dimensions \cite{Iyama2007}. For a given algebra $\Lambda$, its dominant and global dimension are denoted by $\domdim(\Lambda)$ and $\gldim(\Lambda)$, respectively. The $n$-representation dimension of $\Lambda$, denoted by $\repdim_n(\Lambda)$, is defined as follows
$$\repdim_n(\Lambda):=\inf\left\{\gldim\End_{\Lambda}(M)\left|\begin{array}{l}
	M\mbox{ is a generator-cogenerator and }\\
	\domdim \End_{\Lambda}(M)\geq n+1\end{array}\right.\right\}.$$
Recently, a somewhat dual version of representation dimension called {\em rigidity dimension}, denoted by $\rigdim$, is introduced in \cite{chen2021rigidity}
$$\rigdim(\Lambda):=\sup\left\{\domdim\End_{\Lambda}(M)\left|\begin{array}{l}
	M\mbox{ is a generator-cogenerator and }\\
	\gldim \End_{\Lambda}(M)<\infty \end{array}\right.\right\}.$$
Representation dimension is intended to measure how far an algebra is from being representation-finite, while rigidity dimension is introduced for a complete different purpose, it is intended to measure the quality of the best resolutions of $A$, and to compare homological invariants of $A$ {with} its resolutions. For instance, it is proved  in \cite{chen2021rigidity} that, if $\rigdim(A)=n$, then the Hochschild cohomology ring $HH^*(A)$ can have non-nilpotent homogenous generators at degree zero and degrees larger than $n-2$.

Little is known for the precise value of the rigidity dimension of a given algebra. It is even unknown whether this dimension is always finite, although its finiteness does follow if  Yamagata's conjecture (the dominant dimension of an algebra is bounded by a function that depends on the number of isomorphism classes of simple modules) holds true. All Morita algebras \cite{Fang2011,Kernel2013} with dominant dimension $2$ have rigidity dimension 2 and all group algebras have finite rigidity dimension. Chen and Xing \cite{chenxing} calculated the rigidity dimension of certain Hochschild extension of hereditary algebra of type $D$.

%The {difficulties} of calculation lie in the value of dominant dimension and the finiteness of the global dimension of the endomorphism algebra $\End_A(M)$ of a generator-cogenerator $M$.

By M\"{u}ller's criterion \cite{Mueller1968a}, the dominant dimension of the endomorphism algebra of a generator-cogenerator $M$ is precisely the rigidity degree of $M$ plus two, where the {\em rigidity degree} of $M$ is the maximal non-negative integer $n$, or $\infty$, such that $\Ext_A^i(M, M)$ vanishes for all $1\leq i\leq n$. This tells us that rigidity dimension also depends highly on the Ext-structure of the module category.  Concerning infinite rigidity degree, M\"{u}ller \cite{Mueller1968a} proved that the Nakayama Conjecture is true for all Artin algebras if and only if for all Artin algebras, a generator-cogenerator $M$ with rigidity degree infinity would implies that $M$ projective. %Auslander and Reiten then proposed a stronger conjecture: for an Artin algebra $A$, a generator $M$ has rigidity degree infinity implies that $M$ is projective. These conjectures are still open and are related to several other important homological conjectures.

We  focus on the rigidity degrees of modules.  Our main results give formulae for rigidity degrees of all indecomposable modules over a representation-finite self-injective algebra over an algebraically closed field. Let $\Lambda$ be  a non-semisimple representation-finite self-injective algebra of type $(\Delta, u, s)$, where $\Delta$ is a Dynkin diagram with $r$ vertices.  Let $m_{\Delta}$ be the smallest positive integer such that all paths of length $m_{\Delta}$ in the mesh category $k(\Delta)$ are zero, and let $h_{\Delta}=m_{\Delta}+1$ be the Coxeter number. Then the stable Auslander-Reiten quiver of $\Lambda$ is $\mathbb{Z}\Delta/\langle\tau^n\phi\rangle$, where $n=u m_{\Delta}$ and $\phi$ is  an automorphism of $\mathbb{Z}\Delta$ with a fixed vertex. It is quite surprising to us that the rigidity degrees of indecomposable $\Lambda$-modules are closely related to certain combinatorics arising from the Euclidean algorithm for integers $h^*_{\Delta}$ and $n$ (see, Theorem \ref{theo-typeA}, Theorem \ref{theo-typeD} and Theorem \ref{theo-typeE} below), where $h^*_{\Delta}$ is $h_{\Delta}$ for type $A$ and  $h_{\Delta}/2$ for types $D$ and $E$.

\medskip
This paper is organized as follows.
In Section \ref{sect-preliminaries}, we recall some basic definitions and facts.
In Section \ref{sect-combinatorics}, we study combinatorics arising from Euclidean algorithm, introduce weighted Fibonacci sequences and develop Proposition \ref{prop-rm-range} which is crucial for the later proofs.
Explicit formulae of rigidity degrees of indecomposable modules {for} type $A, D$ and $E$ are given in Sections 4-6, respectively. The main results are Theorem \ref{theo-typeA}, \ref{theo-typeD} and \ref{theo-typeE}. Finally, in Section 7, the rigidity dimensions of certain algebras of types $A$ and $E$ are calculated.

\section{Preliminaries}\label{sect-preliminaries}
Throughout this paper, all algebras are connected, non-semisimple and finite dimensional  over an algebraically closed field $k$. We use $\mathbb{N}$ to represent the set of positive integers. For an algebra $\Lambda$, $\modcat{\Lambda}$ denotes the category of all left $\Lambda$-modules; $\stmodcat{\Lambda}$ denotes the stable module category of $\modcat{\Lambda}$. The syzygy and cosyzygy operators of $\Lambda$-mod are denoted by $\Omega_{\Lambda}$ and $\Omega^-_{\Lambda}$. Let $M$ be an $\Lambda$-module, we denote the smallest full subcategory of $\Lambda$-mod containing direct sums and direct summands of $M$ by $\add(M)$.

\begin{Def}\label{def-domdim}
  Let $\Lambda$ be an Artin algebra, and let
  $$0\ra {}_{\Lambda}\Lambda\ra I^0\ra I^1\ra\cdots $$
  be a minimal injective resolution of ${}_{\Lambda}\Lambda$. The {\em dominant dimension}, denoted by $\domdim \Lambda$, is defined to be the largest integer $d\geq 0$ (or $\infty$) such that $I^i$ is projective for all $i<d$ (or $\infty$). For a module $M$ over an algebra $\Lambda$, its {\em rigidity degree}, denoted by $\rd(M)$, is defined as the maximal integer $n>0$ (or $\infty$) such that $\Ext_{\Lambda}^i(M, M)$ vanishes for all $1\leq i\leq n$.
\end{Def}

The connection between rigidity degree and dominant dimension is provided in \cite{Mueller1968a} due to  M\"{u}ller.
\begin{Theo}[\cite{Mueller1968a}]
\label{relat-rig-dom}
 Let $\Lambda$ be an algebra and $M$ a generator-cogenerator of $\Lambda$-mod. Then the dominant dimension of the endomorphism algebra $\End _{\Lambda}(M)$ is precisely  $\rd(M)+2$.
\end{Theo}

\medskip
For a quiver $B$ without loops, the translation quiver  $\mathbb{Z}B$ introduced by Riedtmann \cite{Riedtmann1980a} is defined as follows.
The vertices are $(m,x)$ where $m$ is an integer and $x$ is a vertex of $B$. The arrows of the quiver $\mathbb{Z}B$ are as follows. Each arrow $x\ra y$ in $B$ forms arrows $(m,x)\ra (m,y)$ and $(m,y)\ra (m-1,x)$ for all integers $m$. Note that if the underlying graph of $B$ is a tree without multiple edges, then the translation quiver $\mathbb{Z}B$ is independent of the orientation in $B$. For example, if $B$ is the quiver $x\ra y$, then $\mathbb{Z}B$ is as follows.
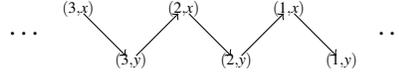
\begin{center}
	\begin{tikzpicture}
	\pgfmathsetmacro{\ul}{0.7};
	\draw[->] (1.1*\ul,4.9*\ul)--(1.9*\ul,4.1*\ul);
	\draw[->] (2.1*\ul,4.1*\ul)--(2.9*\ul,4.9*\ul);
	\draw[->] (3.1*\ul,4.9*\ul)--(3.9*\ul,4.1*\ul);
	\draw[->] (4.1*\ul,4.1*\ul)--(4.9*\ul,4.9*\ul);
	\draw[->] (5.1*\ul,4.9*\ul)--(5.9*\ul,4.1*\ul);
	\fill[fill=black]  (5*\ul,5*\ul)  node[] {${\scriptscriptstyle (1,x)}$};
	\fill[fill=black]  (6*\ul,4*\ul)  node[] {${\scriptscriptstyle (1,y)}$};
		\fill[fill=black]  (3*\ul,5*\ul)  node[] {${\scriptscriptstyle (2,x)}$};
	\fill[fill=black]  (4*\ul,4*\ul)  node[] {${\scriptscriptstyle (2,y)}$};
		\fill[fill=black]  (1*\ul,5*\ul)  node[] {${\scriptscriptstyle (3,x)}$};
	\fill[fill=black]  (2*\ul,4*\ul)  node[] {${\scriptscriptstyle (3,y)}$};
		\fill[fill=black]  (7*\ul,4.5*\ul)  node[] {$\cdots$};
		\fill[fill=black]  (0*\ul,4.5*\ul)  node[] {$\cdots$};
	\end{tikzpicture}
\end{center}
The {\em translation} $\tau$ on $\mathbb{Z}B$ sends $(m, x)$ to $(m+1, x)$. An automorphism $g$ of $\mathbb{Z}B$ is a quiver automorphism which commutes with the translation $\tau$. For a vertex $v$ in $\mathbb{Z}B$, denote by $v^-$ the set of immediate predecessors of $v$ and by $v^+$ the set of immediate successors of $v$. A group $G$ of automorphisms of $\mathbb{Z}B$ is admissible if each orbit of $G$ meets $\{v\}\cup v^+$ in at most one vertex and meets $\{v\}\cup v^-$ in at most one vertex for each vertex $v$ in $\mathbb{Z}B$. In this case, the orbits of $G$ form a translation quiver $\mathbb{Z}B/G$: the vertices are the orbits of $G$, and there is an arrow from $Gv$ to $Gw$ precisely when $Gw\cap v^+\neq\emptyset$, and the translation is given by sending $Gv$ to $G\tau(v)$ which is the same as $\tau Gv$ since all elements of $G$ commute with $\tau$.

\medskip
For a Dynkin quiver $\Delta$ ($A_r, D_r, E_6, E_7, E_8$) with $r$ vertices, let $m_{\Delta}$ be the smallest positive integer such that all paths of length $m_{\Delta}$ in the mesh category $k(\Delta)$ are zero, and let $h_{\Delta}=m_{\Delta}+1$ be the Coxeter number.  It is well-known that $m_{\Delta}=r$ for type $A$, $2r-3$ for type $D$, and $11,17,29$ for type $E_6, E_7$ and $E_8$ respectively.

\medskip
Let $\Lambda$ be an indecomposable representation-finite self-injective non-semisimple algebra, and let $\Gamma_s(\Lambda)$ be the stable AR-quiver of $\Lambda$. It is well-known that there is a Dynkin quiver $\Delta$ such that $\Gamma_s(\Lambda)=\mathbb{Z}\Delta/G$ for some admissible group $G$ (see \cite{Riedtmann1980a}). The group $G$ is generated by an automorphism $\tau^n\phi$, where $n$ is a positive integer and $\phi$ is an automorphism of $\mathbb{Z}\Delta$ with a fixed vertex. Let $s$ be the order of $\phi$. Then $(\Delta, n/m_{\Delta}, s)$ is called the type of $\Lambda$. The complete list of types is as  follows (\cite{BLR}, see also \cite{Asashiba1999}).

\medskip
\begin{itemize}
	\item $(A_{r}, n/r, 1)$, $r,n\in\mathbb{N}$;
	\item $(A_{2p+1}, u, 2)$, $p, s\in \mathbb{N}$;
	\item $(D_r, u, 1)$, $r, u\in\mathbb{N}, r\geq 4$;
	\item $(D_{3w},u/3,1)$, $w, u\in\mathbb{N}, w\geq 2,  3\nmid u$;
	\item $(D_r, u, 2)$, $r,u\in\mathbb{N}, r\geq 4$;
	\item $(D_4, u, 3)$, $u\in\mathbb{N}$;
	\item $(E_r, u, 1)$, $r=6,7,8$, $u\in\mathbb{N}$;
	\item $(E_6, u, 2)$, $u\in\mathbb{N}$.
\end{itemize}

\medskip
Let $\pi: \mathbb{Z}\Delta\lra \mathbb{Z}\Delta/G$ be the natural morphism of translation quivers sending each vertex $v$ to its orbit $Gv$ under $G$. Identifying $\mathbb{Z}\Delta/G$ with $\Gamma_s(\Lambda)$, $\pi$ can be viewed as a morphism from $\mathbb{Z}\Delta$ to $\Gamma_s(\Lambda)$.  The automorphism $\Omega$ of $\Gamma_s(\Lambda)$
induced by the syzygy functor $\Omega_{\Lambda}: \stmodcat{\Lambda}\lra\stmodcat{\Lambda}$ lifts to an automorphism $\omega$ of $\mathbb{Z}\Delta$, that is, $\Omega\pi=\pi\omega$. For each indecomposable $\Lambda$-module $X$, we define
$$H^-(X):=\{Y\in \Gamma_s(\Lambda)|\stHom_{\Lambda}(Y, X)\neq 0\}.$$
$$H^+(X):=\{Y\in \Gamma_s(\Lambda)|\stHom_{\Lambda}(X, Y)\neq 0\}$$
There are natural isomorphisms
$$\stHom_{\Lambda}(Y,\tau X)\cong D\Ext^1_{\Lambda}(X, Y)\cong D\stHom_{\Lambda}(\Omega_{\Lambda}X, Y),$$
where $\tau$ is the Auslander-Reiten translation.  This implies that $H^+(\Omega_{\Lambda}X)=H^-(\tau X)$. For a vertex $x\in \mathbb{Z}\Delta$, one can define combinatorially a set of vertices $H^-(x)$  such that $\pi$ induces a bijection between $H^-(x)$ and $H^-(\pi(x))$.  For the precise construction of $H^-(x)$, we refer to   \cite[4.4.2]{Iyama2005c}. $H^+(x)$ can be then defined as $H^-(\omega^{-1}\tau x)$. For each vertex $x$ of $\mathbb{Z}\Delta$, let $X:=\pi(x)$ be the corresponding indecomposable $A$-module. We define
$$\SE(X):=\{i>0|\Ext_A^i(X,X)\neq 0\}.$$
Then it is easy to see that
$$\begin{aligned}
\SE(X) & =\{i>0|\Ext_A^i(X,X)\neq 0\}\\
& = \{i>0|\stHom_A(\Omega^iX,X)\neq 0\}\\
& = \{i>0|\Omega^iX\in H^-(X)\}
\end{aligned}$$
Similarly we define
$$\SE_G(x):=\{i>0|G\omega^i(x)\cap H^-(x)\neq\emptyset\}$$
Then by definition $\rd(X)=\inf\SE(X)-1$ if $\SE(X)$ is not empty and $\rd(X)=\infty$ otherwise. The notion $\rd_G(x)$ can be similarly defined.

\medskip
The following will be used frequently in our later proofs.
\begin{Lem}
\label{lem-SE-basic-property}
Keep the notations above. The following hold.

$(1)$. $\SE_G(x)=\SE(X)$.

$(2)$.  $\SE_G(x)=\SE_G(\tau(x))$, $\SE_G(x)=\SE_G(\omega (x))$.

$(3)$. $\rd_{G}(x)=\rd(X)$.
\end{Lem}
\begin{proof}
Note that all elements in $G$ commute with $\tau$ and $\omega$. Since $\pi H^{-}(x)=H^{-}(X)$, the lemma follows easily.
\end{proof}

This lemma reduces the question of finding rigidity degrees  to a combinatorial problem on $\mathbb{Z}\Delta$ together with the information of $G$.
It is closely related to the combinatorics we will develop in the next section.

\section{Combinatorics from the Euclidean algorithm}\label{sect-combinatorics}
Give two positive integers $m$ and $n$, we denote by $\rem{m}_{n}$ the remainder of $m$ modulo $n$, that is, a non-negative integer less than $n$ which is  congruent to $m$ modulo $n$. The combinatorics we need in this paper is to determine the range of $\rem{rm}_{n}$
for  positive integers $r$. The main result of this section is Proposition \ref{prop-rm-range} which will be used frequently in later proofs.

\subsection{Weighted Fibonacci sequences}
Given a sequence $\numseq{a}: a_s, a_{s+1}, \cdots, a_r$ of positive integers ($s\leq r$), we define recursively
\begin{equation*}
\Fb_l(\mathbf{a}):=\left\{
    \begin{array}{ll}
      0, & l=s-2; \\
      1, & l=s-1; \\
      a_{l}\Fb_{l-1}(\mathbf{a})+\Fb_{l-2}(\mathbf{a}), & s\leq l\leq r.
    \end{array}
  \right.
\end{equation*}
The sequence $\Fb_{s-2}(\numseq{a}), \Fb_{s-1}(\numseq{a}), \cdots, \Fb_r(\numseq{a})$ is called the {\em weighted Fibonacci sequence} with {\em weight sequence} $\numseq{a}$. This can be written in matrix form as
$$\begin{bmatrix}
\Fb_{l-1}(\mathbf{a})\\
\Fb_l(\mathbf{a})
\end{bmatrix}=\begin{bmatrix}
0&1\\
1& a_l
\end{bmatrix}\begin{bmatrix}
\Fb_{l-2}(\mathbf{a})\\
\Fb_{l-1}(\mathbf{a})
\end{bmatrix},$$
where $s\leq l\leq r$. For simplicity, we write
$$A(a):=\left[\begin{array}{cc}0&1\\1&a\end{array}\right]$$
for each number $a$. For each sequence $\numseq{a}: a_{s}, a_{s+1},\cdots,a_{r}$, the sequence $a_{s+1},\cdots,a_{r}$ obtained by removing the starting number is denoted by $\numseq{a'}$.  Then it is straightforward to check that
$$\begin{bmatrix}
\Fb_{l}(\numseq{a'})& \Fb_l(\mathbf{a})\\
\Fb_{l+1}(\numseq{a'}) & \Fb_{l+1}(\mathbf{a})
\end{bmatrix}=A(a_{l+1})\cdots A(a_s).$$
Note that if we shift the sequence, say a sequence $\numseq{b}$ is obtained from $\numseq{a}$ by setting $b_{i}=a_{i+t}$ for some fixed $t$ and for all $i$, then $\Fb_{i}(\numseq{b})=\Fb_{i+t}(\numseq{a})$ for all $i$.

\subsection{The range of the remainders}
For the rest of this section, we fix two positive integers $m$ and $n$, and set
$$s_{-1}:=m, \quad s_{0}:=n$$
The Euclidean algorithm gives rise to a sequence of equations:
\begin{align*}
 s_{-1} & =k_0s_0+s_1\\
    s_0 &=k_{1}s_{1}+s_{2}\\
    & \cdots\cdots\\
    s_d&=k_{d+1}s_{d+1}+s_{d+2}\\
     s_{d+2}&=0
\end{align*}
where $0<s_i<s_{i-1}$ for all $1\leq i\leq d+1$.
Here we get a sequence of positive integers
$$k_1, \cdots, k_{d+1}$$
which is called the {\em weight sequence} of $m$ and $n$, denoted by $\wseq{m,n}$, $d+1$ is called the length of the weight sequence, denoted by $|\wseq{m,n}|$, the sequence $s_1, \cdots, s_{d+1}$ is called the {\em remainder sequence} of $m,n$.

\medskip
For the rest of this section, we write $\kk$ for $\wseq{m,n}$ for simplicity.

\medskip
Note that the equations given by the Euclidean algorithm can also be written as matrices multiplications, namely, for $0\leq l\leq d+1$,
$$\begin{bmatrix} s_{l+1} & s_l\end{bmatrix}\cdot A(k_l)=\begin{bmatrix} s_l & s_{l-1}\end{bmatrix}$$
$$ A(k_l)\cdot\begin{bmatrix}s_{l-1}\\-s_l\end{bmatrix}=\begin{bmatrix}-s_{l}\\s_{l+1}\end{bmatrix}$$

\medskip
The rest of this section is devoted to studying $\rem{rm}_{n}$
for positive integers $r$.  Since $s_{-1}=m,\, s_{0}=n$, and  $s_{-1}=k_{0}s_{0}+s_{1}$,  we have $\rem{rm}_{n}= \rem{rs_{1}}_{s_{0}}$.
The following lemma deals with the case that $r$ is $\Fb_i(\numseq{k})$ for some $i$.

\begin{Lem}
\label{lemma-Bls1-rem}
Keep the notations above. For $1\leq l\leq d+2$, we have $$\Fb_{l-1}(\numseq{k})s_1\equiv (-1)^{l-1}s_l\pmod{s_0}.$$
\end{Lem}
\begin{proof}
This is clear for $l=1$. Now assume that $l>1$. Recall that we denote by $\numseq{k}'$ the sequence $k_{2}, \cdots, k_{d+1}$. Then
\begin{align*}
\begin{bmatrix}
\Fb_{l-1}(\numseq{k'})& \Fb_{l-1}(\numseq{k})\\
\Fb_{l}(\numseq{k'})& \Fb_{l}(\numseq{k})
\end{bmatrix}
\begin{bmatrix}
-s_{0}\\ s_{1}
\end{bmatrix}&=A(k_{l})\cdots A(k_{1})
\begin{bmatrix}
-s_{0}\\ s_{1}
\end{bmatrix}\\
&= -A(k_{l})\cdots A(k_{2})
\begin{bmatrix}
-s_{1}\\ s_{2}
\end{bmatrix}\\
&=\cdots = (-1)^{l}
\begin{bmatrix}
-s_{l}\\ s_{l+1}
\end{bmatrix}.
\end{align*}
It follows that
$\Fb_{l-1}(\numseq{k})s_{1}-\Fb_{l-1}(\numseq{k'})s_{0}=(-1)^{l-1}s_{l}$ and the lemma follows.
 \end{proof}
A particular case of the above lemma is
$$\Fb_{d+1}(\numseq{k})s_{1}\equiv (-1)^{d+1}s_{d+2}\equiv 0 \pmod{s_{0}}.$$
This means that we only need to consider positive integers less than $\Fb_{d+1}(\numseq{k})$. The next lemma expresses such integers $r$ as linear combinations of $\Fb_i(\numseq{k})$, which is useful when we consider the remainder $\rem{rs_{1}}_{s_{0}}$.
\begin{Lem}
\label{lemma-rs1-rem}
Let $1\leq l\leq d+1$, and $0<r\leq \Fb_l(\numseq{k})$. Then $r$ can be written as
$$r=\sum_{i=1}^l\lambda_i\Fb_{i-1}(\numseq{k})$$
such that $0\leq \lambda_i\leq k_i$ for all $1\leq i\leq l$ and $\lambda_1>0$. Furthermore, we have
$$rs_{1}\equiv \sum_{i=1}^l(-1)^{i-1}\lambda_is_i\pmod{s_0}.$$
\end{Lem}
\begin{proof}
For simplicity, we write $\Fb_{i}$ for $\Fb_{i}(\numseq{k})$ for all $i$.  We use induction on $l$. If $l=1$, then each $0<r\leq \Fb_{1}$ is of the form $\lambda_{1}\Fb_{0}$ with $0<\lambda_{1}\leq k_{1}$.
Assume now that $l>1$. Since
$$\Fb_{l}=k_{l}\Fb_{l-1}+\Fb_{l-2},$$
each $0<r\leq \Fb_{l}$ is of the form $r=p\Fb_{l-1}+q$ with $0\leq p\leq k_{l}$ and $0\leq q<\Fb_{l-1}$.

If $q=0$, then $p>0$ since $r>0$. By induction, we can assume that $\Fb_{l-1}=\sum_{i=1}^{l-1}\lambda_{i}\Fb_{i-1}$ with $0<\lambda_{1}\leq k_{1}$ and $0\leq \lambda_{i}\leq k_{i}$ for $2\leq i\leq l-1$. Thus
$$r=p\Fb_{l-1}=(p-1)\Fb_{l-1}+\sum_{i=1}^{l-1}\lambda_{i}\Fb_{i-1}=\sum_{i=1}^{l}\lambda_{i}\Fb_{i-1}, \quad (\lambda_{l}:=p-1)$$
as desired.

If $q>0$, then $q<\Fb_{l-1}$, and we can assume that $q=\sum_{i=1}^{l-1}\lambda_{i}\Fb_{i-1}$ with $0<\lambda_{1}\leq k_{1}$ and $0\leq \lambda_{i}\leq k_{i}$ for $2\leq i\leq l-1$. Defining $\lambda_{l}:=p$, we have $r=\sum_{i=1}^{l}\lambda_{i}\Fb_{i-1}$ with the desired properties.

Together with Lemma \ref{lemma-Bls1-rem}, the rest of the lemma follows.
\end{proof}

The following lemma justifies the expression of $r$ as a linear combination of $\Fb_{i}(\numseq{k})$.
\begin{Lem}
\label{lemma-rem-range}
Suppose that $1\leq l\leq d+1$ and $0\leq \lambda_i\leq k_i$ for all $1\leq i\leq l$. If
   $\lambda_1>0$, then
$$0\leq \sum_{i=1}^l(-1)^{i-1}\lambda_is_i\leq s_0.$$
Moreover, setting $r=\sum_{i=1}^{l}\lambda_{i}\Fb_{i-1}$, either of the equalities holds if and only if $r=\Fb_{d+1}$.
\end{Lem}
\begin{proof}
Clearly, we have
$$\lambda_{1}s_{1}-\sum_{i\mbox{ \small is even}}\lambda_{i}s_{i}\leq  \sum_{i=1}^l(-1)^{i-1}\lambda_is_i\leq \sum_{i\mbox{ \small is odd}}\lambda_{i}s_{i}.$$
Since $\lambda_{i}\leq k_{i}$ for all $i$ and $\lambda_{1}>0$, we further get
$$s_{1}-\sum_{i\mbox{ \small is even}}k_{i}s_{i}\leq  \sum_{i=1}^l(-1)^{i-1}\lambda_is_i\leq \sum_{i\mbox{ \small is odd}}k_{i}s_{i}.$$
Now the term of left hand side is
$s_{1}-(s_{1}-s_{3})-(s_{3}-s_{5})-\cdots =s_{t}\geq 0$
with $t-1$ the maximal even integer $\leq l$. Thus, the equality on the left hand side holds if and only if $\lambda_{1}=1$, $s_{t}=0$, $\lambda_{i}=k_{i}$ for all even $i$ and zero for all the other odd $i$, that is, $t=d+2$ and $r=\Fb_{0}+k_{2}\Fb_{1}+k_{4}\Fb_{3}+\cdots+k_{d+1}\Fb_{d}=\Fb_{d+1}$.

The right hand side is
$(s_{0}-s_{2})+(s_{2}-s_{4})+\cdots =s_{0}-s_{t}\leq s_{0},$
where $t-1$ is the maximal odd integer $\leq l$.  Again, the equality on this side holds if and only if $s_{t}=0$, $\lambda_{i}=k_{i}$ for all odd $i$ and zero for even $i$. This happens precisely when $t=d+2$ and $r=k_{1}\Fb_{0}+k_{3}\Fb_{2}+\cdots+k_{d+1}\Fb_{d}=\Fb_{d+1}$.
\end{proof}

\medskip
 An immediate consequence of Lemma \ref{lemma-rs1-rem} and \ref{lemma-rem-range} is that, for $0<r<\Fb_{d+1}(\numseq{k})$, the remainders $\rem{rs_{1}}_{s_{0}}$ are non-zero and pairwise distinct. Particularly $\Fb_{d+1}(\kk)\leq s_0$.   The following proposition is technically crucial in our later proofs.

\begin{Prop}
\label{prop-rm-range}
Let $m, n$ be positive integers, and let $\kk=\wseq{m,n}$ be the weight sequence with $|\kk|=d+1$. Suppose that $0<l\leq d+1$, and  $0<r\leq \Fb_l(\mathbf{k})$ (respectively, $0<r<\Fb_l(\mathbf{k})$) when $l\leq d$ is odd (respectively, $l$ is even or $l=d+1$). Then

\begin{itemize}
\item[$(1)$] $\rem{rm}_{n}\geq s_l$ and $\rem{(r-1)m}_{n}\leq n-s_l$.
\item[$(2)$]  If $l$ is odd, then  $\rem{rm}_{n}=s_l$ if and only if $r=\Fb_{l-1}(\mathbf{k})$, and  $\rem{(r-1)m}_{n}=n-s_l$  if and only if $d$ is even, $l=d+1$ and $r=\Fb_{d+1}(\mathbf{k})-\Fb_d(\mathbf{k})+1$.
\item[$(3)$] If $l$ is even, then $\rem{rm}_{n}= s_l$ if and only if $d$ is odd, $l=d+1$ and $r=\Fb_{d+1}(\mathbf{k})-\Fb_d(\mathbf{k})$, and  $\rem{(r-1)m}_{n}=n-s_l$ if and only if $r=\Fb_{l-1}(\mathbf{k})+1$.
\end{itemize}
\end{Prop}

\begin{proof}
Let $s_{-1}=m, s_0=n$, $s_1, \cdots, s_{d+1}$ be the remainder sequence. For simplicity, we write $\Fb_{l}$ for $\Fb_{l}(\mathbf{k})$ throughout this proof.

Note that $\rem{rm}_n=\rem{rs_1}_{s_0}$ for all integers $r$. The proposition is clear for the case $r=1$. In the following, we assume that $1<r<\Fb_{d+1}$. By Lemma \ref{lemma-rs1-rem}, $r$ and $r-1$ can be written as
$$r=\sum_{i=1}^{l}\lambda_{i}\Fb_{i-1},\quad r-1=\sum_{i=1}^{l}\mu_{i}\Fb_{i-1}$$
with $0<\lambda_{1}, \mu_{1}\leq k_{1}$ and $0 \leq \lambda_{i},\mu_{i}\leq k_{i}$ for all $2\leq i\leq l$. Since $r<\Fb_{d+1}$, by Lemma \ref{lemma-rs1-rem} and \ref{lemma-rem-range}, we see that
$$\rem{rs_{1}}_{s_{0}}=\sum_{i=1}^{l}(-1)^{i-1}\lambda_{i}s_{i}, \quad \rem{(r-1)s_{1}}_{s_{0}}=\sum_{i=1}^{l}(-1)^{i-1}\mu_{i}s_{i}.$$
%it follows from Lemma \ref{lemma-rs1-rem} that
%$$rs_{1}\equiv \sum_{i=1}^{l}(-1)^{i-1}\lambda_{i}s_{i}\pmod{ s_0}.$$
Let us consider the case $l$ is odd. Then
\begin{align*}
\rem{rs_{1}}_{s_{0}}&= \sum_{i=1}^{l}(-1)^{i-1}\lambda_{i}s_{i}\geq s_{1}-\sum_{i\mbox{ \small is even}} \lambda_{i}s_{i}\\
&\geq s_{1}-k_{2}s_{2}-k_{4}s_{4}-\cdots-k_{l-1}s_{l-1}\\
&=s_{1}-(s_{1}-s_{3})-\cdots-(s_{l-2}-s_{l})=s_{l}
\end{align*}
The equality holds if and only if $\lambda_{1}=1$, $\lambda_{i}=k_{i}$ for all even $i$, and $\lambda_{i}=0$ for all odd $i>1$. Equivalently
$$r=\sum_{i=1}^{l}\lambda_{i}\Fb_{i-1}=\Fb_{0}+k_{2}\Fb_{1}+k_{4}\Fb_{3}+\cdots+k_{l-1}\Fb_{l-2}= \Fb_{l-1}$$
Next, we consider $r-1$, which is
$$r-1=\sum_{i=1}^{l}\mu_{i}\Fb_{i-1}.$$
Note that $\Fb_{l}=k_{l}\Fb_{l-1}+k_{l-2}\Fb_{l-3}+\cdots+k_{3}\Fb_{2}+k_{1}\Fb_{0}$.  It can not happen that $\mu_{i}=0$ for all even $i$ and $\mu_{i}=k_{i}$ for all odd $i$ since $r-1\neq \Fb_{l}$. Let $t\leq l$ be such that $\mu_{t}>0$ when $t$ is even or $\mu_{t}<k_{t}$ when $t$ is odd.
\begin{align*}
\rem{(r-1)s_{1}}_{s_{0}} &=\mu_{1}s_{1}-\mu_{2}s_{2}+\cdots+\mu_{l}s_{l}\\
&\leq k_{1}s_{1}+k_{3}s_{3}+\cdots+k_{l}s_{l}-s_{t}\\
&=(s_{0}-s_{2})+(s_{2}-s_{4})+\cdots+(s_{l-1}-s_{l+1})-s_{t}\\
& =s_{0}-s_{l+1}-s_{t}\\
&\leq s_{0}-s_{t}\leq s_{0}-s_{l}
\end{align*}
The equality holds if and only if $s_{l+1}=0$, $t=l$, $\mu_{l}=k_{l}-1$, $\mu_{i}=k_{i}$ for all odd $i<l$, and $\mu_{i}=0$ for all even $i$. Equivalently, $l=d+1$, and
$r-1=k_{1}\Fb_{0}+k_{3}\Fb_{2}+\cdots+(k_{d+1}-1)\Fb_{d}=\Fb_{d+1}-\Fb_{d}$,
that is, $r=\Fb_{d+1}-\Fb_{d}+1$.

Assume now that $l$ is even and $r<\Fb_{l}$. In this case $$\Fb_{l}=k_{l}\Fb_{l-1}+k_{l-2}\Fb_{l-3}+\cdots+k_{2}\Fb_{1}+\Fb_{0}.$$
It cannot happen that $\lambda_{1}=1$, $\lambda_{i}=0$ for all odd $i>1 $ and $\lambda_{i}=k_{i}$ for all even $0<i\leq l$ since $r\neq \Fb_{l}$. Let $0<t\leq l$ be such that $\lambda_{t}>1$ when $t=1$, $\lambda_{t}>0$ when $t>1$ is odd or $\lambda_{t}<k_{t}$ when $t$ is even. Then

\begin{align*}
\rem{rs_{1}}_{s_{0}}&=\sum_{i=1}^{l}(-1)^{i-1}\lambda_{i}s_{i}\\
&\geq s_{1}-k_{2}s_{2}-k_{4}s_{4}-\cdots-k_{l}s_{l}+s_{t}\\
&=s_{l+1}+s_{t}\geq s_{t}\geq s_{l}
\end{align*}
The equality holds if and only if $s_{l+1}=0$, $t=l$, $\lambda_{l}=k_{l}-1$, $\lambda_{i}=0$ for all odd $1<i<l$,  $\lambda_{i}=k_{i}$ for all even $0<i<l$, and $\lambda_{1}=1$. Equivalently, $l=d+1$, and $r=\Fb_{0}+k_{2}\Fb_{1}+k_{4}\Fb_{3}+\cdots+(k_{d+1}-1)\Fb_{d}=\Fb_{d+1}-\Fb_{d}$. Finally
\begin{align*}
\rem{(r-1)s_{1}}_{s_{0}}&=\sum_{i=1}^{l}(-1)^{i-1}\mu_{i}s_{i}\leq \sum_{i\mbox{ \small is odd}}k_{i}s_{i}\\
&=s_{0}-s_{2}+s_{2}-s_{4}+\cdots+s_{l-1}-s_{l}\\
&=s_{0}-s_{l}.
\end{align*}
The equality holds if and only if $\mu_{i}=0$ for all even $i$ and $\mu_{i}=k_{i}$ for all odd $i$, that is,
$r-1=k_{1}\Fb_{0}+k_{3}\Fb_{2}+\cdots+k_{l-1}\Fb_{l-2}=\Fb_{l-1}$.
\end{proof}

\section{Rigidity degrees of indecomposable modules: type $A$}\label{rig-nakayama}

In this section,  we shall present explicit formulae for rigidity degrees of indecomposable modules over self-injective algebras of type $A$. Suppose that $\Lambda$ is a representation-finite self-injective algebra of type $A$. Its stable Auslander-Reiten quiver is $\mathbb{Z}A_{m-1}/G$, where $G$ is an admissible automorphism group of $\mathbb{Z}A_{m-1}$. Note that the Coxeter number is $(m-1)+1=m$ in this case.

We coordinate the  translation quiver $\mathbb{Z}A_{m-1}$  as follows.
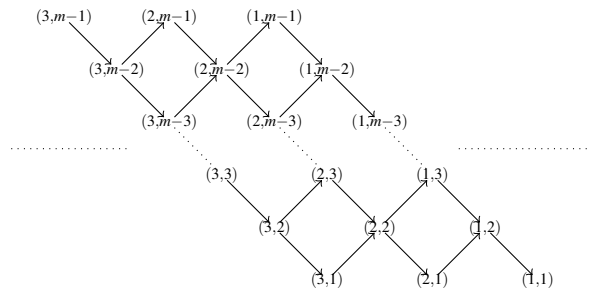
\begin{center}
\begin{tikzpicture}
\pgfmathsetmacro{\ul}{0.7};
\draw[->] (1.1*\ul,4.9*\ul)--(1.9*\ul,4.1*\ul);
\draw[->] (2.1*\ul,3.9*\ul)--(2.9*\ul,3.1*\ul);
\draw[dotted] (3.1*\ul,2.9*\ul)--(3.9*\ul,2.1*\ul);
\draw[->] (4.1*\ul,1.9*\ul)--(4.9*\ul,1.1*\ul);
\draw[->] (5.1*\ul,0.9*\ul)--(5.9*\ul,0.1*\ul);
\draw[->] (3.1*\ul,4.9*\ul)--(3.9*\ul,4.1*\ul);
\draw[->] (4.1*\ul,3.9*\ul)--(4.9*\ul,3.1*\ul);
\draw[dotted] (5.1*\ul,2.9*\ul)--(5.9*\ul,2.1*\ul);
\draw[->] (6.1*\ul,1.9*\ul)--(6.9*\ul,1.1*\ul);
\draw[->] (7.1*\ul,0.9*\ul)--(7.9*\ul,0.1*\ul);
\draw[->] (5.1*\ul,4.9*\ul)--(5.9*\ul,4.1*\ul);
\draw[->] (6.1*\ul,3.9*\ul)--(6.9*\ul,3.1*\ul);
\draw[dotted] (7.1*\ul,2.9*\ul)--(7.9*\ul,2.1*\ul);
\draw[->] (8.1*\ul,1.9*\ul)--(8.9*\ul,1.1*\ul);
\draw[->] (9.1*\ul,0.9*\ul)--(9.9*\ul,0.1*\ul);
\draw[->] (2.1*\ul,4.1*\ul)--(2.9*\ul,4.9*\ul);
\draw[->] (3.1*\ul,3.1*\ul)--(3.9*\ul,3.9*\ul);
\draw[->] (4.1*\ul,4.1*\ul)--(4.9*\ul,4.9*\ul);
\draw[->] (5.1*\ul,3.1*\ul)--(5.9*\ul,3.9*\ul);
\draw[->] (5.1*\ul,1.1*\ul)--(5.9*\ul,1.9*\ul);
\draw[->] (6.1*\ul,0.1*\ul)--(6.9*\ul,0.9*\ul);
\draw[->] (7.1*\ul,1.1*\ul)--(7.9*\ul,1.9*\ul);
\draw[->] (8.1*\ul,0.1*\ul)--(8.9*\ul,0.9*\ul);
\fill[fill=black]  (1*\ul,5*\ul)  node[] {${\scriptscriptstyle (3,m-1)}$};
\fill[fill=black]  (2*\ul,4*\ul) node[] {${\scriptscriptstyle (3,m-2)}$};
\fill[fill=black]  (3*\ul,3*\ul) node[] {${\scriptscriptstyle (3,m-3)}$};
\fill[fill=black]  (4*\ul,2*\ul) node[] {${\scriptscriptstyle (3,3)}$};
\fill[fill=black]  (5*\ul,1*\ul) node[] {${\scriptscriptstyle (3,2)}$};
\fill[fill=black]  (6*\ul,0*\ul) node[] {${\scriptscriptstyle (3,1)}$};
\fill[fill=black]  (3*\ul,5*\ul) node[] {${\scriptscriptstyle (2,m-1)}$};
\fill[fill=black]  (4*\ul,4*\ul) node[] {${\scriptscriptstyle (2,m-2)}$};
\fill[fill=black]  (5*\ul,3*\ul) node[] {${\scriptscriptstyle (2,m-3)}$};
\fill[fill=black]  (6*\ul,2*\ul) node[] {${\scriptscriptstyle (2,3)}$};
\fill[fill=black]  (7*\ul,1*\ul) node[] {${\scriptscriptstyle (2,2)}$};
\fill[fill=black]  (8*\ul,0*\ul) node[] {${\scriptscriptstyle (2,1)}$};
\fill[fill=black]  (5*\ul,5*\ul) node[] {${\scriptscriptstyle (1,m-1)}$};
\fill[fill=black]  (6*\ul,4*\ul) node[] {${\scriptscriptstyle (1,m-2)}$};
\fill[fill=black]  (7*\ul,3*\ul) node[] {${\scriptscriptstyle (1,m-3)}$};
\fill[fill=black]  (8*\ul,2*\ul) node[] {${\scriptscriptstyle (1,3)}$};
\fill[fill=black]  (9*\ul,1*\ul) node[] {${\scriptscriptstyle (1,2)}$};
\fill[fill=black]  (10*\ul,0*\ul) node[] {${\scriptscriptstyle (1,1)}$};
\draw[dotted] (0,2.5*\ul)--(2.2*\ul,2.5*\ul);
\draw[dotted] (8.5*\ul,2.5*\ul)--(11*\ul,2.5*\ul);
\end{tikzpicture}
\end{center}
There are two classes of types:

\medskip
(1) $(A_{m-1}, n/(m-1), 1)$, $m,n\in\mathbb{N}, m\geq 2$ , and

(2) $(A_{2p+1}, u, 2)$, $p, u\in\mathbb{N}$.

\medskip
The main result of this section is the following theorem.
\begin{Theo}\label{theo-typeA}
	Suppose that $\Lambda$ is a representation-finite self-injective algebra of type $(A_{m-1},u,s)$. Set $$M=m,\quad N=u(m-1),\mbox{ when }s=1, \mbox{ and }$$
	$$M=u(m-1)+m/2,\quad  N=2u(m-1), \mbox{ when }s=2.$$
	Let $\kk =\wseq{M,N}$ be the weight sequence, and let $s_i$, $\Fb_i(\kk)$, $-1\leq i\leq |\kk|$ be the remainder sequence and the corresponding weighted Fibonacci sequence respectively. Suppose that  $X$ is an indecomposable $\Lambda$-module corresponding to the vertex $(x,t)$ in $\mathbb{Z}A_{m-1}$ with $t\leq m/2$. Then the rigidity degree $\rd(X)$ is listed in Table \ref{tab-typeA}.
	\renewcommand{\arraystretch}{1.7}
	\begin{table}[h]
		\centering
	\begin{tabular}{l|l}
	\hline
	 \quad\quad\quad  $\rd(X)$ & condition \\
	 \hline
	  $\frac{2}{s}\Fb_l({\bf k})-1$, & $s_{l+1}<t<s_l$, $l$ is even, or $l=|\wseq{}|$; \\
	 $\frac{2}{s}\Fb_l({\bf k})$, &  $s_{l+1}\leq t\leq s_l$, $l<|\wseq{}|$ is odd;\\
	 $\frac{2}{s}(\Fb_{|\kk|}({\bf k})-\Fb_{|\kk|-1}({\bf k}))$, & $|\wseq{}|$ is odd and $t=s_{|\kk|}\leq m/2$.\\
	\hline
	\end{tabular}
	\caption{\label{tab-typeA} Rigidity degrees: Type $A$}
\end{table}
	\end{Theo}
The rest of this section is devoted to giving a proof of Theorem \ref{theo-typeA}.

\medskip
For each vertex $(x,t)$ of $\mathbb{Z}A_{m-1}$, one can check that $H^{-}(x,t)$ consists of  vertices in the  rectangle  and its boundary below.
$$
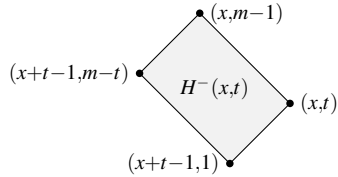
\begin{tikzpicture}
\pgfmathsetmacro{\ul}{0.4};
\draw [-,fill=gray!10] (5*\ul,0) node[circle,fill=black,inner sep=1pt] {} node[left] {$\scriptstyle (x+t-1,1)$}--(7*\ul,2*\ul) node[circle,fill=black,inner sep=1pt] {} node[right] {${\scriptstyle (x,t)}$} --(4*\ul,5*\ul) node[circle,fill=black,inner sep=1pt] {} node[right] {$\scriptstyle (x,m-1)$}  --(2*\ul,3*\ul) node[circle,fill=black,inner sep=1pt] {} node[left] {$\scriptstyle (x+t-1,m-t)$} -- cycle; % H- rectangle
\node at (4.5*\ul,2.5*\ul) {$\scriptstyle H^-(x,t)$};
%\draw [-] (9*\ul,0) node[circle,fill,inner sep=1pt] {} node[right] {$\scriptstyle (x,1)$}--(7*\ul,2*\ul) --(10*\ul,5*\ul) node[circle,fill,inner sep=1pt] {} node[right] {$\scriptstyle (x+t-m+1,m-1)$}  --(12*\ul,3*\ul) node[circle,fill,inner sep=1pt] {} node[right] {$\scriptstyle (x+t-m+1,m-t)$} -- cycle; % H- rectangle
%\node at (9.5*\ul,2.5*\ul) {$\scriptstyle H^+(x,t)$};
%\draw [-,fill=gray!10] (5*\ul,0) node[circle,fill,inner sep=1pt] {} node[left] {$(x+t-1,1)$}--(7*\ul,2*\ul) node[circle,fill,inner sep=1pt] {} node[right] {$ (x,t)$} --(4*\ul,5*\ul) node[circle,fill,inner sep=1pt] {} node[right] {$(x,m-1)$}  --(2*\ul,3*\ul) node[circle,fill,inner sep=1pt] {} node[left] {$(x+t-1,m-t)$} -- cycle;
%\draw [-,fill=gray!10] (5*\ul,0)--(7*\ul,2*\ul)--(4*\ul,5*\ul)--(2*\ul,3*\ul)-- cycle;
%\fill[fill=black]  (5*\ul,0) circle (1pt) node[left] {$(x+t-1,1)$};
%\fill[fill=black]  (4*\ul,5*\ul) circle (1pt) node[right] {$(x,m-1)$};
%\fill[fill=black]  (2*\ul,3*\ul) circle (1pt) node[left] {$(x+t-1,m-t)$};
%\node at (4.5*\ul,2.5*\ul) {$\scriptstyle H^-(x,t)$};
%\fill[fill=black]  (2*\ul,3*\ul) circle (1pt);
%\draw [-,fill=gray!10] (9*\ul,0)--(7*\ul,2*\ul)--(10*\ul,5*\ul)--(12*\ul,3*\ul)-- cycle;
%\fill[fill=black]  (9*\ul,0) circle (1pt) node[right] {$(x,1)$};
%\fill[fill=black]  (10*\ul,5*\ul) circle (1pt) node[right] {$(x-m+t+1,m-1)$};
%\fill[fill=black]  (12*\ul,3*\ul) circle (1pt) node[right] {$(x-m+t+1,m-t)$};
%\node at (9.5*\ul,2.5*\ul) {${\scriptstyle H^+(x,t)}$};
%\fill[fill=black]  (7*\ul,2*\ul) circle (1pt) node[right] {$ (x,t)$};
\end{tikzpicture}$$
and $\omega(x,t)=(x+t,m-t)$. Thus, for each integer $k$,  we have
 $$\omega^{2k}(x,t)=(x+km,t),\quad \omega^{2k+1}(x,t)=(x+km+t,m-t)$$

\begin{Lem}
\label{lemma-HfA}
For $0<t\leq m/2$, and $t\leq t'\leq m-t$, $(y,t')\in H^-(0,t)$ if and only if  $0\leq y<t$.
\end{Lem}
\begin{proof}
This is obvious from the above picture.
\end{proof}

Since the vertices in the same $\langle\tau,\omega\rangle$-orbit have the same rigidity degree, it suffices to consider $\rd_G(0,t)$ for $0<t\leq m/2$. For simplicity, we write $\SE_G(t)$ for $\SE_G(0,t)$ and $\rd_G(t)$ for $\rd_G(0,t)$.

\medskip
The group $G$ is cyclic, and is generated by $\tau^n$ for type $(A_{m-1}, u, 1)$ with $n=u(m-1)$, and is generated by $\tau^n\omega$ for type $(A_{2p+1}, u,2)$ with $n=u(2p+1)-(p+1)$. We shall divide the proof of Theorem \ref{theo-typeA} into two cases.
 \subsection{Case I: $G=\langle\tau^n\rangle$}

This happens for type $(A_{m-1}, n/(m-1),1)$. The following proposition collects some basic properties of $\SE_{G}(t)$.
\begin{Prop}
\label{prop-inSE}
Keep the notations above. Let $0<t\leq m/2$. Then

\begin{itemize}
	\item[$(1)$] $2k\in\SE_G(t)$ if and only if $\rem{km}_n<t$.
	\item[$(2)$]  $2k+1\in\SE_G(t)$ if and only if $\rem{km}_{n}\geq n-t$.
	\item[$(3)$] If $t>1$, then $\SE_{G}(t-1)\subseteq\SE_{G}(t)$. In particular, $\rd_G(t)\leq\rd_G(t-1)$.
\end{itemize}
\end{Prop}
\begin{proof}
(1). Since $\omega^{2k}(0, t)=(km,t)$, one gets that $2k\in\SE_{G}(t)$ if and only if there is some integer $i$ such that $0\leq km+in< t$ by Lemma \ref{lemma-HfA}, that is, $\rem{km}_{n}<t$. This proves (1).

$(2)$. Recall that $\omega^{2k+1}(0,t)=(km+t, m-t)$. By Lemma \ref{lemma-HfA} again,
we see that $2k+1\in \SE_{G}(t)$ if and only if there is an integer $i$ such that $0\leq km+t+in<t$, that is, $\rem{km}_{n}\geq n-t$.

$(3)$ follows immediately from $(1)$ and $(2)$.
\end{proof}

\begin{Def}
	A positive integer $t\leq m/2$ is called an {\bf endpoint}  if $t=1$ or $\rd_G(t)<\rd_G(t-1)$.
\end{Def}

From the statement (3) of Proposition \ref{prop-inSE}, we see that when $t$ goes from $m/2$ down to $1$, the rigidity degree $\rd_G(t)$ is (not strictly) increasing.  Thus, to find $\rd_G(t)$ for $t\leq m/2$, it suffices to find all the endpoints and their rigidity degrees. The following proposition gives a characterization of endpoints.

\begin{Prop}
\label{prop-cutoff-pts}
Let  $0<t\leq m/2$ and let $r>0$. Then
\begin{itemize}
	\item[$(1)$] $t$ is an endpoint of rigidity degree $2r-1$ if and only if, $\rd_G(t)\geq 2r-1$ and $\rem{rm}_n=t-1$.
	\item[$(2)$] $t$ is an endpoint of rigidity degree $2(r-1)$ if and only if, $\rd_G(t)\geq 2(r-1)$ and $\rem{(r-1)m}_n=n-t$.
\end{itemize}
\end{Prop}

\begin{proof}
Suppose that $t=1$. In this case $t$ is by definition an endpoint. If $\rd_G(1)=2r-1$, then $2r\in\SE_{G}(1)$, and thus $\rem{rm}_{n}<1$ and hence $\rem{rm}_{n}=0=t-1$ by Proposition \ref{prop-inSE} (1). If $\rd_G(1)=2(r-1)$, then $2r-1\in\SE_{G}(1)$. Hence $\rem{(r-1)m}_{n}\geq n-1$ and thus $\rem{(r-1)m}_{n}=n-1=n-t$ by Proposition \ref{prop-inSE}  (2).

Now assume that $t>1$. Then, under the hypothesis $\rd_{G}(t)\geq 2r-1$, $t$ is an endpoint of rigidity degree $2r-1$ if and only if $2r\in\SE_{G}(t)$ and $2r\not\in\SE_{G}(t-1)$, which is equivalent to $\rem{rm}_{n}=t-1$ by Proposition \ref{prop-inSE} (1).  Under the hypothesis $\rd_{G}(t)\geq 2(r-1)$, $t$ is an endpoint of rigidity degree $2(r-1)$ if and only if $2r-1\in\SE_{G}(t)$ and $2r-1\not\in\SE_{G}(t-1)$. This is equivalent to $\rem{(r-1)m}_{n}=n-t$ by Proposition \ref{prop-inSE} (2).
\end{proof}

\iffalse
Now we are ready to give the formula for the rigidity degrees of vertices of $\mathbb{Z}A_{m-1}/\langle\tau^n\rangle$.  The Euclidean algorithm gives rise to a sequence of equalities.
$$\begin{aligned}
m & =k_0n+s_1\\
n&=k_1s_1+s_2\\
&\cdots\cdots\\
s_d &= k_{d+1}s_{d+1}+s_{d+2}\\
s_{d+2}&=0
\end{aligned}$$
Let $\mathbf{k}$ be the weight sequence $$k_{1}, k_{2}, \cdots, k_{d+1}.$$

\begin{Theo}
Keep the notations above.
	
$(1)$. Assume that $t\leq m/2$ is a positive integer and $l\geq 0$. Then
$$\rd_G(t)=
\begin{cases}
2\Fb_l(\mathbf{k})-1, & s_{l+1}<t<s_l, l \mbox{ is even, or }l=d+1; \\
2\Fb_l(\mathbf{k}), & s_{l+1}\leq t\leq s_l, l\leq d\mbox{ is odd}.
\end{cases}$$

$(2)$. If $d$ is even and $s_{d+1}\leq m/2$, then
$$\rd_G(s_{d+1})=2(\Fb_{d+1}(\mathbf{k})-\Fb_d(\mathbf{k})).$$
\label{thm-rdAtaun}
\end{Theo}
\fi
\begin{proof}[{\bf Proof of Theorem \ref{theo-typeA}($s=1$)}]
In this case, $M=m, N=n$. For simplicity, we write $\Fb_l$ for $\Fb_l(\mathbf{k})$ and $d+1=|\kk|$ in this proof. What we need to prove is, for each $1\leq t\leq m/2$,
$$\rd_G(t)=\begin{cases}
	2\Fb_l-1, & s_{l+1}<t<s_l, l\mbox{ is even, or }l=d+1;\\
	2\Fb_l, &  s_{l+1}\leq t\leq s_l, l<d+1\mbox{ is odd};\\
	2(\Fb_{d+1}-\Fb_{d}), & d\mbox{ is even and }t=s_{d+1}\leq m/2.\\
\end{cases}$$
If $l=-1$ and $s_{l+1}=n\leq m/2$, then $n$ is an endpoint of rigidity degree $0$ by Proposition \ref{prop-cutoff-pts} (2).  Hence $\rd_G(t)=2\Fb_{-1}=0$ for all $s_0\leq t\leq m/2$. Suppose that $0\leq l\leq d+1$ and $s_{l+1}<t<s_l$. For each $r<\Fb_l$, by Proposition \ref{prop-rm-range}, one gets $\rem{rm}_n\geq s_l>t$ and
$\rem{(r-1)m}_{n}\leq n-s_{l}<n-t$. It follows that $2r\not\in\SE_{G}(t)$ and $2r-1\not\in\SE_{G}(t)$ for all $r<\Fb_{l}$ by Proposition \ref{prop-inSE}. Hence $\rd_{G}(t)\geq 2\Fb_{l}-2$. To show that $\rd_{G}(t)\geq 2\Fb_{l}-1$, it suffices to prove that $2\Fb_{l}-1\not\in\SE_{G}(t)$, or equivalently,
$$\rem{(\Fb_{l}-1)m}_{n}<n-t.$$

Assume that $l$ is even.  Then $\rem{(\Fb_{l}-1)m}_{n}\equiv (-1)^{l}s_{l+1}-s_{1}\equiv s_{l+1}-s_{1}\pmod{n}$. If $l=0$, then $\rem{(\Fb_{l}-1)m}_{n}=0= n-s_{l}<n-t$.
If $l\geq 2$, then $-n\leq s_{l+1}-s_{1}<0$ and thus
$$\begin{aligned}
\rem{(\Fb_{l}-1)m}_{n} & =n+s_{l+1}-s_{1}\\
& =n -(k_{2}s_{2}+s_3)+s_{l+1}\\
& \leq n-k_{2}s_{2}\\
&\leq n-s_{2}\leq n-s_{l}<n-t.
\end{aligned}$$

Now assume that $l=d+1$ and $t<s_{d+1}$. Since $s_{d+1}$ is a greatest common divisor of $m$ and $n$, the remainder $\rem{rm}_{n}$ is always a multiple of $s_{d+1}$. Hence $\rem{rm}_{n}\leq n-s_{d+1}$ for all integers $r$. Particularly,
$$\rem{(\Fb_{l}-1)m}_{n}=\rem{(\Fb_{d+1}-1)m}_{n}\leq n-s_{d+1}<n-t.$$

Altogether, we have shown that for $l$ is even or $l=d+1$, and $s_{l+1}<t<s_{l}$, there is an inequality $\rd_{G}(t)\geq 2\Fb_{l}-1$. Moreover, $\rem{\Fb_{l}m}_{n}=(-1)^{l}s_{l+1}=s_{l+1}$ when $l$ is even; $\rem{\Fb_{d+1}m}_{n}=(-1)^{d+1}s_{d+2}=0$.  This means that $s_{l+1}+1$ is an endpoint of rigidity degree $2\Fb_{l}-1$ by Proposition \ref{prop-cutoff-pts} (1).
Together with Proposition \ref{prop-inSE} (3), we deduce that
$\rd_{G}(t)=2\Fb_{l}-1$  for all $s_{l+1}<t<s_{l}$.

\medskip
Now assume that $l\leq d$ is odd and $s_{l+1}\leq t\leq s_{l}$. For any $r\leq \Fb_{l}$, by Proposition \ref{prop-rm-range}, we have $\rem{rm}_{n}\geq s_{l}$ and $\rem{(r-1)m}_{n}\leq n-s_{l}$. In the later inequality, equality possibly holds only if  $l=d+1$ which contradicts to our assumption $l\leq d$. Hence $\rem{(r-1)m}_{n}< n-s_{l}$. Since $t\leq s_{l}$, it follows that $\rem{rm}_{n}\geq s_{l}\geq t$ and $\rem{(r-1)m}_{n}<n-s_{l}\leq n-t$ for all $r\leq \Fb_{l}$. By Proposition \ref{prop-inSE}, this means that $2r\not\in\SE_{G}(t)$ and $2r-1\not\in\SE_{G}(t)$ for all $r\leq \Fb_{l}$. Hence $\rd_{G}(t)\geq 2\Fb_{l}$ for all $s_{l+1}\leq t\leq s_{l}$. Moreover,
$$
\rem{\Fb_{l}m}_{n} = \rem{(-1)^{l}s_{l+1}}_{n}
 =\rem{-s_{l+1}}_{n}
 =n-s_{l+1}.
$$
By Proposition \ref{prop-cutoff-pts} (2), one deduces that $s_{l+1}$ is an endpoint of rigidity degree $2\Fb_{l}$. Together with Proposition \ref{prop-inSE} (3), we conclude that $\rd_{G}(t)=2\Fb_{l}$ for all $s_{l+1}\leq t\leq s_{l}$.

Finally, we consider the case that $d$ is even and $s_{d+1}\leq m/2$. Then $l=d+1$ is odd. By Proposition \ref{prop-rm-range}, $\rem{rm}_{n}\geq s_{d+1}$ and  $\rem{(r-1)m}_{n}\leq n-s_{d+1}$ for all $0<r<\Fb_{l}$. Moreover, $\rem{(r-1)m}_{n}= n-s_{d+1}$ if and only if $r=\Fb_{d+1}-\Fb_{d}+1$. Hence $\rem{rm}_{n}\geq s_{d+1}$ and $\rem{(r-1)m}_{n}<n-s_{d+1}$ for all $r\leq \Fb_{d+1}-\Fb_{d}$. It follows that $2r, 2r-1\not\in\SE_{G}(s_{d+1})$ for all $0<r\leq \Fb_{d+1}-\Fb_{d}$, and thus $\rd_{G}(s_{d+1})\geq 2(\Fb_{d+1}-\Fb_{d})$. However, the fact
$$\rem{(\Fb_{d+1}-\Fb_{d})m}_{n}=n-s_{d+1}$$
implies that $2(\Fb_{d+1}-\Fb_{d})+1\in\SE_{G}(s_{d+1})$. Hence $\rd_{G}(s_{d+1})=2(\Fb_{d+1}-\Fb_{d})$.
\end{proof}

\subsection{Case II: $G=\langle\tau^n\omega\rangle$} \label{subsect-mob}
The case happens in type $(A_{m-1}, u, 2)$ and $n=u(m-1)-m/2$. Then
$$M:=n+m,\quad N:=2n+m.$$

\begin{Prop}
	\label{prop-mob-SE}
	Let $t\leq m/2$ be a positive integer and let $r$ be a positive integer. Then $r\in\SE_{G}(t)$ if and only if $\rem{rM}_{N}<t$ or $\rem{(r-1)M}_{N}\geq N-t$.
	\end{Prop}
	\begin{proof}
	
	$r\in\SE_{G}(t)$ if and only if
	 $G\omega^{r}(0,t)\cap H^{-}(0, t)\neq\emptyset$, if and only if there is an integer $k$ such that $$(\tau^{n}\omega)^{2k}\omega^{r}(0,t)\in H^{-}(0,t), \mbox{ or } (\tau^{n}\omega)^{2k+1}\omega^{r}(0,t)\in H^{-}(0,t). \quad (\star)$$
	Note that $\omega^{2}=\tau^{m}$, which will be used frequently in this proof.
	
	(1) Assume that $r=2l$ is even. Then
	$$(\tau^{n}\omega)^{2k}\omega^{r}(0,t)  =\tau^{2kn}\tau^{m(l+k)}(0,t) = (kN+lm,t), $$
	$$
	(\tau^{n}\omega)^{2k+1}\omega^{r}(0,t)  =\tau^{2kn+n}\tau^{m(l+k)}\omega(0,t) = (kN+lm+n+t,m-t).$$
	Thus $r\in\SE_G(t)$ if and only if there is some integer $k$ such that $0\leq kN+lm<t$ or $0\leq kN +lm+n+t<t$. Equivalently, $\rem{lm}_{N}<t$ or $\rem{lm+n+t}_{N}<t$. Now $\rem{lm}_{N}=\rem{2lM}_{N}=\rem{rM}_{N}$ and $\rem{lm+n+t}_{N}<t$ if and only if $\rem{lm+n}_{N}\geq N-t$. However $(2l-1)M\equiv lm+n \pmod{N}$. Hence $r=2l\in \SE_{G}(t)$ if and only if $\rem{rM}_{N}<t$ or $\rem{(r-1)M}_{N}\geq N-t$.
	
	(2) Now suppose that $r=2l+1$ is odd. By calculation, one gets
	$$(\tau^{n}\omega)^{2k}\omega^{r}(0,t)=(kN+lm+t,m-t)$$
	$$(\tau^{n}\omega)^{2k+1}\omega^{r}(0,t)=(kN+lm+m+n,t)$$
	and deduces that $r\in \SE_{G}(t)$ if and only if $\rem{lm+t}_{N}<t$ or $\rem{lm+m+n}_{N}<t$.  Equivalently, $\rem{lm}_{N}\geq N-t$ or $\rem{lm+m+n}_{N}<t$. Finally $\rem{lm}_{N}=\rem{2lM}_{N}$ and $(2l+1)M\equiv lm+m+n\pmod{N}$. Hence $r\in\SE_{G}(t)$ if and only if $\rem{rM}_{N}<t$ or $\rem{(r-1)M}\geq N-t$.
	\end{proof}
	
	An immediate consequence is the following.
	
	\begin{Koro}
	\label{koro-mob-rd-decreasing}
	Suppose $0<t\leq t'\leq m/2$. Then $\SE_{G}(t)\subseteq\SE_{G}(t')$ and $\rd_{G}(t)\geq\rd_{G}(t')$.
	\end{Koro}
	
Thus, $\rd_G(t)$ is increasing when $t$ goes from $m/2$ down to $1$, and we can similarly call a positive integer $t\leq m/2$ an endpoint if $t=1$ or $\rd_G(t)<\rd_G(t-1)$.

	\begin{Prop}
	\label{prop-mob-cut-off}
	Suppose that $0<t\leq m/2$. Then $t$ is an endpoint of rigidity degree $r$ if and only if $\rd_{G}(t)\geq r$ and either $\rem{(r+1)M}_{N}=t-1$ or $\rem{rM}_{N}=N-t$ holds.
	\end{Prop}
\begin{proof}
Suppose that $\rd_G(t)=r$. By definition, $t$ is an endpoint if and only if either $t=1$ or $\rd_G(t)<\rd_G(t-1)$. Equivalently, $t=1$ or $t>1, r+1\notin\SE_G(t-1)$.

Suppose that $t=1$. Then $r+1\in\SE_G(1)$. By Proposition \ref{prop-mob-SE}, this means that $\rem{(r+1)M}_N<1$ or  $\rem{rM}_N\geq N-1$. This forces that $\rem{(r+1)M}_N=0=t-1$ or  $\rem{rM}_N=N-1=N-t$.

Suppose that $t>1$. Then $r+1\notin \SE_G(t-1)$ means $\rem{(r+1)M}_N\geq t-1$ and $\rem{rM}_N< N-(t-1)$. Since $\rd_G(t)=r$, we have $r+1\in\SE_G(t)$, and thus $\rem{(r+1)M}_N<t$ or  $\rem{rM}_N\geq N-t$. Hence $r+1\notin \SE_G(t-1)$ is equivalent to the condition  $\rem{(r+1)M}_{N}=t-1$ or $\rem{rM}_{N}=N-t$ in this case.
\end{proof}
\iffalse 	
	Now we set
	$$s_{-1}=M, s_0=N, s_1=M.$$
	The Euclidean algorithm gives
	$$\begin{aligned}
		s_{-1}& =0\cdot s_0+s_1\\
		s_0& =k_1s_1+s_2\\
		&\cdots\cdots \\
		s_d &=k_{d+1}s_{d+1}
	\end{aligned}$$
	width $s_1>s_2>\cdots>s_{d+1}>0$. Let $s_{d+2}=0$, and let
	$${\bf k}: k_1, k_2, \cdots, k_{d+1}.$$
\fi

	%Set $s_{0}=n$ and $s_{-1}=M$.  As in the case $G=\langle\tau^{n}\rangle$, the Euclidean algorithm gives $m=k_{0}s_{0}+s_{1}$,  $s_{i}=k_{i+1}s_{i+1}+s_{i+2}$ for $i=0,\cdots,d$ and $s_{d+2}=0$. Then
	%$$\begin{aligned}
	%N&= 1\cdot M+s_{0}\\
	%M&= (1+k_{0})s_{0}+s_{1}\\
	%n&=k_{1}s_{1}+s_{2}\\
	%&\cdots\cdots\\
	%s_{d}&=k_{d+1}s_{d+1}.
	%\end{aligned}$$
	%We record the sequence
	%$$\bar{\bf k}: 1, (1+k_{0}), k_{1}, \cdots, k_{d+1},$$
	%where $1$ is $\bar{k}_{-1}$.
\iffalse
	\begin{Theo}	
	$(1)$. Assume that $t\leq m/2$ is a positive integer and $l\geq 1$. Then
	$$\rd_G(t)=
	\begin{cases}
	\Fb_l({\bf k})-1, & s_{l+1}<t<s_l, l \mbox{ is even, or }l=d+1; \\
	\Fb_l({\bf k}), & s_{l+1}\leq t\leq s_l, l\leq d\mbox{ is odd}.
	\end{cases}$$
	$(2)$. If $d$ is even and $s_{d+1}\leq m/2$, then
	$$\rd_G(s_{d+1})=\Fb_{d+1}({\bf k})-\Fb_d({\bf k}).$$
	\label{thm-rdAomegataun}
	\end{Theo}
\fi 	
	\begin{proof}[{\bf Proof of Theorem \ref{theo-typeA}($s=2$)}]
	For simplicity, we write $\Fb_l$ for $\Fb_l(\kk)$, and write $d+1=|\kk|$. Note that $t\leq m/2<M=s_1$.
	For each $1\leq l\leq d+1$, $0<r<\Fb_l$ and $t<s_l$, one has
	$$\rem{rM}_N\geq s_l>t\mbox{ and } \rem{(r-1)M}_N\leq N-s_l<N-t$$
	by Proposition \ref{prop-rm-range}. Together with Proposition \ref{prop-mob-SE}, this implies that $r\notin\SE_G(t)$. Hence $\rd_G(t)\geq \Fb_l-1$ for all $t<s_l$.
	
	If $l$ is even, or $l=d+1$, then by Lemma \ref{lemma-Bls1-rem}
	$$\rem{\Fb_lM}_N\equiv (-1)^{l}s_{l+1}=\left\{\begin{array}{ll}
	s_{l+1},& l \mbox{ is even,} \\
	0,& l=d+1.
	\end{array}\right. \pmod{N}$$
	Actually in both cases we get $\rem{\Fb_lM}_N=s_{l+1}$ since $s_{d+2}=0$. It follows from Proposition \ref{prop-mob-cut-off} that {$t=s_{l+1}+1$} is an endpoint of rigidity degree $\Fb_l-1$. Altogether, we have
	$$\Fb_l-1\leq \rd_G(t)\leq \rd_G(s_l+1)=\Fb_l-1$$	
	for all $s_{l+1}<t<s_l$. Here the second $\leq$ follows from Corollary \ref{koro-mob-rd-decreasing}. Hence $\rd_G(t)=\Fb_l-1$ for all $s_{l+1}<t<s_l$ if $l$ is even or $l=d+1$.
	
	Now assume that $l\leq d$ is odd, $r\leq \Fb_l$ and $t\leq s_l$. If follows from Proposition \ref{prop-rm-range} that
	$$\rem{rM}_N\geq s_l\geq t\mbox{ and }\rem{(r-1)M}_N\leq N-s_l\leq N-t.$$
	Moreover, the equality $\rem{(r-1)M}_N= N-s_l$ holds only if $l=d+1$ which is excluded by our assumption $l\leq d$. Hence
$\rem{rM}_N\geq t\mbox{ and }\rem{(r-1)M}_N<N-t$ for all $r\leq \Fb_l$ and $t\leq s_l$, and thus $\rd_G(t)\geq \Fb_l$ for all $t\leq s_l$. Moreover
	$$\rem{\Fb_lM}_N\equiv (-1)^ls_{l+1}\equiv N-s_{l+1} \pmod{N}.$$
	By Proposition \ref{prop-mob-cut-off}, $t=s_{l+1}$ is an endpoint of rigidity degree $\Fb_l$, and it follows that $\rd_G(t)=\Fb_l$ for all $s_{l+1}\leq t\leq s_l$.
	
	Now the only missing case is that $l=d+1$ is odd and $t=s_{d+1}\leq m/2$. In this case
	$$\rem{(\Fb_{d+1}-\Fb_d)M}_N=N-s_{d+1}$$
	by Proposition \ref{prop-rm-range}. This means that
	$$\Fb_{d+1}-\Fb_d+1\in \SE_G(s_{d+1})$$
	by Proposition \ref{prop-mob-SE} and thus $\rd_G(t)\leq \Fb_{d+1}-\Fb_d$. Now for $r\leq \Fb_{d+1}-\Fb_d$ which is of course less than $\Fb_{d+1}$. By Proposition \ref{prop-rm-range}, we have
	$$\rem{rM}_N\geq s_{d+1}\mbox{ and }\rem{(r-1)M}_N\leq N-s_{d+1}.$$
	Here the equality $\rem{(r-1)M}_N= N-s_{d+1}$ cannot hold since $r\neq \Fb_{d+1}-\Fb_d+1$. Together with Proposition \ref{prop-mob-SE}, one again has $r\notin\SE_G(s_{d+1})$. Hence
	$\rd_G(s_{d+1})=\Fb_{d+1}-\Fb_d$ in this case.
	\end{proof}

\section{Rigidity degrees of indecomposable modules: type $D$}

For convenience, we assume that the stable Auslander-Reiten quiver of $\Lambda$ is $\mathbb{Z}D_{m+1}/G$ so that $m$ is precisely half of the Coxeter number.  The vertices of $D_{m+1}$ are labelled as follows.
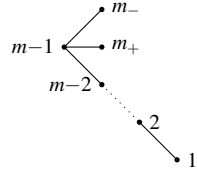
\begin{center}
\begin{tikzpicture}[scale=0.5]
	\draw (1,4)--(0,3)--(1,2);
	\draw[dotted] (1,2)--(2,1);
	\draw (2,1)-- (3,0);
	\draw (0,3)--(1,3);
	\fill[fill=black]  (0,3) circle (2pt) node[left] {${\scriptstyle m-1}$};
	\fill[fill=black]  (1,2) circle (2pt) node[left] {${\scriptstyle m-2}$};
	\fill[fill=black]  (3,0) circle (2pt) node[right] {${\scriptstyle 1}$};
	\fill[fill=black]  (2,1) circle (2pt) node[right] {${\scriptstyle 2}$};
	\fill[fill=black]  (1,3) circle (2pt) node[right] {${\scriptstyle m_{+}}$};
	\fill[fill=black]  (1,4) circle (2pt) node[right] {${\scriptstyle m_-}$};
\end{tikzpicture}
\end{center}
The type of $\Lambda$ is one of the following.
\begin{itemize}
	\item $(D_{m+1}, u, 1), m, u\in \mathbb{N}, m\geq 3$;
	\item $(D_{3w},v/3,1), w,v\in\mathbb{N},w\geq 2, 3\nmid v$;
	\item $(D_{m+1}, u, 2), m, u\in \mathbb{N}, m\geq 3$;
	\item $(D_4, u, 3), u\in\mathbb{N}$.
\end{itemize}
For each type $(D_{m+1}, u, s)$, where $s=1,2,3$, the group $G$ is generated by $\tau^n\phi$, where $n=u(2m-1)$ and $\phi$ is induced by an automorphism of $D_{m+1}$ with order $s$. That is, $\phi$ is the identity map when $s=1$; $\phi$ interchanges $(x, m_+)$ and  $(x, m_-)$ for all $x\in\mathbb{Z}$ when $s=2$. The case $s=3$ happens only for $D_4$ and $\phi$ is induced by the obvious automorphism of $D_4$ of order $3$.

\medskip
The rigidity degrees of indecomposable $\Lambda$-modules can be formulated as the following theorem.

\begin{Theo}\label{theo-typeD}
Suppose that $\Lambda$ is of type $(D_{m+1}, u, s)$. Let $n=u(2m-1)$ and let ${\bf k}=k(m,n)$ be the weight sequence. Then for each indecomposable $\Lambda$-module $X$, corresponding to the vertex $(x,t)$ on $\mathbb{Z}D_{m+1}$, the rigidity degree $\rd(X)$ can be read from Table \ref{tab-typeD}.
	\begin{table}[h]
    \renewcommand{\arraystretch}{1.7}
	\centering
	\begin{tabular}{c|c|ll}
	\hline
	  & vertex $t$ & $\rd_G(t)$ & condition \\
	 \hline
	\multirow{3}{*}{$s\neq 3$} & \multirow{3}{*}{$t<m$} &  $\Fb_l({\bf k})-1$, &  $s_{l+1}<t<s_l$, $l$ is even, or $l=|{\bf k}|$ \\
	& & $\Fb_l({\bf k})$, &  $s_{l+1}\leq t\leq s_l$, $l<|\kk|$ is odd\\
	& &$\Fb_{|\kk|}({\bf k})-\Fb_{|\kk|-1}({\bf k})$, & $|\kk|$ is odd and $t=s_{|\kk|}<m$\\
	\hline
	$s\neq 3, m\geq n$
	& $t=m_{\pm}$ & 0 &\\
	\hline
	\multirow{4}{*}{$s\neq 3, m<n$} & \multirow{4}{*}{$t=m_{\pm}$} &
	$\Fb_1({\bf k})-1$, & $m\mid n,  \Fb_1({\bf k})+n+s$  is odd\\
	&& $2\Fb_1({\bf k})-1$, & $m\mid n,  \Fb_1({\bf k})+n+s$  is even\\
	&& $\Fb_1({\bf k})$, &$m\nmid n,  \Fb_1({\bf k})+n+s$  is even\\
	&&  $\Fb_1({\bf k})+\Fb_2({\bf k})$, & $m\nmid n,  \Fb_1({\bf k})+n+s$  is odd\\
	\hline
	\multirow{5}{*}{$s=3$} & \multirow{3}{*}{$t\neq 2$} &
	$3\Fb_1({\bf k})-1$, & $u\equiv 0 \pmod{3}$\\
	&& $\Fb_1({\bf k})$, & $u\equiv 1 \pmod{3}$\\
	&& $2\Fb_1({\bf k})$, &$u\equiv 2 \pmod{3}$\\
	\cline{2-4}
	& \multirow{2}{*}{$t=2$} &
	$\Fb_1({\bf k})-1$, & $3\mid u$\\
	&& $\Fb_1({\bf k})$, & $3\nmid u$\\
	\hline
	\end{tabular}
	\caption{\label{tab-typeD} Rigidity degrees: Type $D$}
  \end{table}
	\end{Theo}

\medskip
We divide the proof of Theorem \ref{theo-typeD} into three cases:

\medskip
\quad {\bf Case 1}: $s\neq 3$ and $t<m$;

\quad {\bf Case 2}: $s\neq 3$ and $t=m_{\pm}$;

\quad {\bf Case 3}: $s=3$, that is, $\Lambda$ is of type $(D_4, u, 3)$.

\medskip
First we assume that $s\neq 3$.
When $t<m$, the figure of $H^-(x, t)$ in $\mathbb{Z}D_{m+1}$ is as follows.
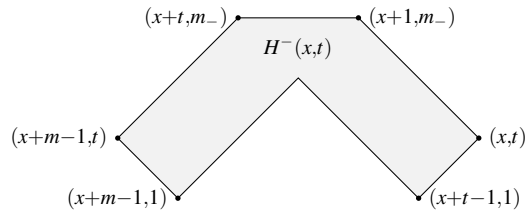
\begin{center}
	\begin{tikzpicture}
	\pgfmathsetmacro{\ul}{0.2};
	\draw [-,fill=gray!10] (-8*\ul,0)--(-12*\ul,4*\ul)--(-4*\ul,12*\ul)--(4*\ul,12*\ul)--(12*\ul,4*\ul)--(8*\ul,0*\ul)--(0*\ul,8*\ul)-- cycle;
	\fill[fill=black]  (-8*\ul,0) circle (1pt) node[left] {$\scriptstyle{(x+m-1,1)}$};
	\fill[fill=black]  (-12*\ul,4 *\ul) circle (1pt) node[left] {$\scriptstyle{(x+m-1,t)}$};
	\fill[fill=black]  (-4*\ul,12*\ul) circle (1pt) node[left] {$\scriptstyle{(x+t,m_{-})}$};
	\fill[fill=black]  (12*\ul,4*\ul) circle (1pt) node[right] {$\scriptstyle{(x,t)}$};
	\fill[fill=black]  (4*\ul,12*\ul) circle (1pt) node[right] {$\scriptstyle{(x+1,m_{-})}$};
	\fill[fill=black]  (8*\ul,0*\ul) circle (1pt) node[right] {$\scriptstyle{(x+t-1,1 )}$};
	\node at (0*\ul,10*\ul) {$\scriptstyle{ H^-(x,t)}$};

	\end{tikzpicture}
	\end{center}
%%% the figure here
The group $G$ is generated by $\tau^n\phi$, where $\phi$ is the identity map or the map interchanging $(x,m_-)$ and $(x, m_+)$. Since $t<m$, we have $\phi(x,t)=(x,t)$ for all $x\in\mathbb{Z}$. Hence $G(x,t)=\{(x+kn, t)\mid k\in\mathbb{Z}\}$. Moreover, from the figure, we have $\omega(x,t)=(x+m,t)$.
\begin{Lem}\label{lem-typeD-low}
Suppose that $\Lambda$ is of type $(D_{m+1}, u, s)$ with $s\neq 3$. Let $n=u(2m-1)$. Then for each $t<m$, $r\in \SE_G(t)$ if and only if $\rem{rm}_n<t$ or $\rem{(r-1)m}_n\geq n-t$.
\end{Lem}
\begin{proof}
By the discussion before the lemma, together with the figure of $H^-(0, t)$, we see that $r\in \SE_G(0,t)$ if and only if there exists integer $k$ such that
$$rm+kn\in [0, t)\cup [m-t, m)$$
This is equivalent to the condition $\rem{rm}_n<t$ or $\rem{(r-1)m}_n\geq n-t$.
\end{proof}

\begin{proof}[{\bf Proof of Theorem \ref{theo-typeD}(Case 1)}]
	Note that Lemma \ref{lem-typeD-low} is similar to Proposition \ref{prop-mob-SE}. Carrying out an identical proof as Subsection \ref{subsect-mob} gives the proof of Theorem \ref{theo-typeD} in Case 1.
\end{proof}

Now we switch to Case 2. Clearly $\rd_G(m_-)=\rd_G(m_+)$. It suffices to consider $\rd_G(m_+)$. For convenience, we define $\sgn(k)$ to be the the symbol "$+$" when $k$ is even, and "$-$" when $k$ is odd.   The figure of $H^-(x, m_+)$ is as follows.
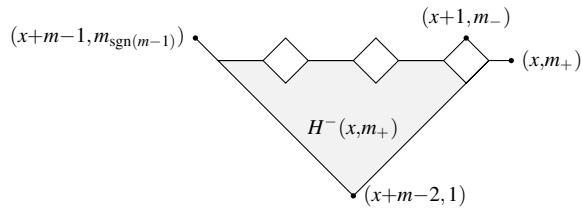
\begin{center}
\begin{tikzpicture}
\pgfmathsetmacro{\ul}{0.3};
\draw [-,fill=gray!10] (0*\ul,6*\ul)--(2*\ul,6*\ul)--(3*\ul,5*\ul)--(4*\ul,6*\ul)--(6*\ul,6*\ul)--(7*\ul,5*\ul)--(8*\ul,6*\ul)--(10*\ul,6*\ul)--(11*\ul,5*\ul)--(12*\ul,6*\ul)--(6*\ul,0*\ul)--cycle;
\draw (-1*\ul,7*\ul)--(0*\ul,6*\ul)--(2*\ul,6*\ul)--(3*\ul,7*\ul)--(4*\ul,6*\ul)--(6*\ul,6*\ul)--(7*\ul,7*\ul)--(8*\ul,6*\ul)--(10*\ul,6*\ul)--(11*\ul,7*\ul)--(12*\ul,6*\ul)--(13*\ul,6*\ul);
\fill[fill=black]  (13*\ul,6*\ul) circle (1pt) node[right] {${\scriptstyle (x,m_+)}$};
\fill[fill=black]  (11*\ul,7*\ul) circle (1pt) node[above] {${\scriptstyle (x+1,\, m_-)}$};
\fill[fill=black]  (-1*\ul,7*\ul) circle (1pt) node[left] {${\scriptstyle (x+m-1,\, m_{\sgn(m-1)})}$};
\fill[fill=black]  (6*\ul,0*\ul) circle (1pt) node[right] {${\scriptstyle (x+m-2,\, 1)}$};
\node at (6*\ul,3*\ul) {${\scriptstyle H^-(x, m_+)}$};
\iffalse
(12*\ul,11*\ul)--(13*\ul,10*\ul)--(14*\ul,11*\ul)--(24*\ul,11*\ul)--(18*\ul,0*\ul)-- cycle;
\fill[fill=black]  (25*\ul,11*\ul) circle (1pt) node[right] {$(x,m_+)$};
\draw (25*\ul,11*\ul)--(24*\ul,11*\ul)--(23*\ul,12*\ul) -- (22*\ul,11*\ul) ;
\fill[fill=black]  (18*\ul,0*\ul) circle (1pt) node[left] {$(x+m-1,1)$};
\fill[fill=black]  (11*\ul,12*\ul) circle (1pt) node[left] {$(x+m-1,m_{(m-)}})$};
\draw (11*\ul,12*\ul) -- (12*\ul,11*\ul);
%\draw [-] (23.4*\ul,11*\ul)--(22.3*\ul,11*\ul);
%\fill[]  (22.3*\ul,11*\ul) circle (1pt) node[left] {$(x,m)$};
\node at (18*\ul,7*\ul) {${\scriptstyle H^-(x,t)}$};
\fi
\end{tikzpicture}
\end{center}
 Moreover, $\omega^r(0, m_+)=(rm, m_{\sgn(rm-r)})=(rm, m_{\sgn(rm+r)})$. We need the following lemma.
\begin{Lem}\label{lem-rem<m}
	Suppose that $n>m$ are positive integers and $\kk=\wseq{m,n}$ is the corresponding weight sequence. Then the following hold.
\begin{itemize}
	\item[$(1)$] If $m\mid n$,  then $\rem{rm}_n<m$ if and only if $r$ is a multiple of $\Fb_1(\kk)$. In this case $\rem{rm}_n=0$.
	\item[$(2)$]  If $m\nmid n$, then, for each $0<r<\Fb_1(\kk)+\Fb_2(\kk)$, $\rem{rm}_n<m$ if and only if $r=p\Fb_1(\kk)+1$ for some integer $1\leq p\leq k_2$.
\end{itemize}
	\end{Lem}
\begin{proof}
For simplicity, we write $\Fb_l$ for $\Fb_l(\kk)$ for each $l$. Let $s_1=m, s_2, \cdots, s_{|\kk|+1}=0$ be the remainder sequence of $m,n$.

If $m\mid n$, then $|\kk|=1$, and $\Fb_1m=k_1m=n\equiv 0\pmod{n}$. By Proposition \ref{prop-rm-range}, $\rem{rm}_n\geq s_1=m$ for all $0<r<\Fb_1$. It follows that $\rem{rm}_n<m$ if and only if $r$ is divided by $\Fb_1$, that is, $r=a\Fb_1$ for some positive integer $a$. Clearly $\rem{a\Fb_1m}_n=0$ in this case.

If $m\nmid n$, then $s_2\neq 0$.  Recall that $\Fb_1=k_1$ and $\Fb_2=k_2\Fb_1+1$. Thus $\Fb_1+\Fb_2-1=k_2k_1+k_1$.  Each positive integer $r<\Fb_1+\Fb_2$ can be written as $r=q+p\Fb_1$ with $1\leq q\leq k_1$ and $0\leq p\leq k_2$.
By Lemma \ref{lemma-rs1-rem} and Lemma \ref{lemma-rem-range}, we deduce that $\rem{rm}_n = qm-ps_2$. Hence $\rem{rm}_n<m$ if and only if $q=1$ and $0<p\leq k_2$, that is, $r=p\Fb_1+1, 0<p\leq k_2$.
\end{proof}

Now we give the proof of Case 2.
\begin{proof}[{\bf Proof of Theorem \ref{theo-typeD}(Case 2)}]
For simplicity, we shall write $\Fb_l$ for $\Fb_l(\kk)$ in the proof.
Let $r$ be a positive integer, then
$r\in\SE_G(m_+)$ if and only if there exists some integer $k$ such that $$(\tau^n\phi)^k\omega^r(0,m_+)\in H^-(0,m_+),$$
Since $\phi$ is the identity map when $s=1$, and $\phi$ interchanges $(x, m_+)$ and  $(x, m_-)$ when $s=2$, the  condition above is equivalent to
$0\leq rm+kn<m$
and $$\sgn(rm+r)=\sgn(rm+kn)\mbox{ when }s=1; \sgn(rm+r+k)=\sgn(rm+kn)\mbox{ when }s=2.$$
Note that it may happen that $m\geq n$, however, only for type  $(D_{3w},1/3,1)$, where $m=3w-1$ and $n=2w-1$. In this case, $m+1-(m-n)=2w$ is even, that is, $\sgn(m+1)=\sgn(m-n)$. Moreover, $0\leq m-n<m$. It follows that $1\in\SE_G(m_+)$ always holds. Hence $\rd_G(m_+)=0$. Now assume that $n>m$. The condition above is further equivalent to $\rem{rm}_n<m$ and $T^s_r$ is even, where
$$T^s_r:=\begin{cases}
	rm+r-\rem{rm}_n, & s=1;\\
	rm+r+({\rem{rm}_n-rm})/{n}-\rem{rm}_n, &  s=2.\\
\end{cases}$$

If $m\mid n$, then $\rem{rm}_n<m$ if and only if $r=i\Fb_1$ for some positive integer $i$ by Lemma \ref{lem-rem<m}, and in this case $i\Fb_1m=in\equiv 0\pmod{n}$. Moreover,
$$T^1_{i\Fb_1}=i\Fb_1(m+1)=i(n+\Fb_1), T^2_{i\Fb_1}=i(n+\Fb_1-1),$$
If $n+\Fb_1+s$ is odd, then $T^s_{\Fb_1}=n+\Fb_1-s+1$ is even, this implies that $\Fb_1\in \SE_G(m_+)$, and thus $\rd_G(m_+)=\Fb_1-1$.
If $n+\Fb_1+s$ is even, then the condition above does not hold for $i=1$, but holds for $i=2$. It follows that $\Fb_1\notin\SE_G(m_+)$ and $2\Fb_1\in\SE_G(m_+)$. Hence $\rd_G(m_+)=2\Fb_1-1$.

Assume now that $m\nmid n$. By Lemma \ref{lem-rem<m}, for each positive integer $r<\Fb_1+\Fb_2$, $\rem{rm}_n<m$ if and only if $r=p\Fb_1+1, 0<p\leq k_2$.
Let us check whether $p\Fb_1+1\in \SE_G(m_+)$. It remains to check whether $T^s_{p\Fb_1+1}$ is even.
Using the fact that $\Fb_1m+s_2=n$ and $\Fb_2m-s_3=k_2n$, it is straightforward to check that $$T^s_{p\Fb_1+1}=p(n+\Fb_1-s+1)+1,$$
which is even if and only if both $p$ and $n+\Fb_1-s+1$ are odd. If $n+\Fb_1+s$ is even, taking $p=1$, we get that $\Fb_1+1\in\SE_G(m_+)$ and thus $\rd_G(m_+)=\Fb_1$.

 Assume that  $n+\Fb_1+s$ is odd. Then $T^s_{p\Fb_1+1}$ can  never be even, and $p\Fb_1+1\notin\SE_G(m_+)$ for all $0<p\leq k_2$. It follows that $\rd_G(m_+)\geq \Fb_1+\Fb_2-1$. Note that $(\Fb_1+\Fb_2)m\equiv -s_2+s_3 \pmod{n}$. Thus $\rem{(\Fb_1+\Fb_2)m}_n=n-s_2+s_3=k_1m+s_3\geq m$. Hence $\Fb_1+\Fb_2\notin\SE_G(m_+)$. Since $s_2>s_3\geq 0$, we deduce that
$$\rem{(\Fb_1+\Fb_2+1)m}_n=m-s_2+s_3<m.$$
Moreover, $T^s_{\Fb_1+\Fb_2+1}=(k_2+1)(n+\Fb_1-s+1)+2$
which is even for $s=1,2$ since $n+\Fb_1+s$ is odd. Therefore $\Fb_1+\Fb_2+1\in\SE_G(m_+)$, and $\rd_G(m_+)=\Fb_1+\Fb_2$. This finishes the proof of Case 2.
\end{proof}

Finally, let us consider the type $(D_4, u, 3)$. In this case, $m=3$ and $n=5u>3$. Without loss of generality, we assume that $\phi$ is induced by the $3$-cycle $\sigma=(1,m_-, m_+)$.
\begin{proof}[{\bf Proof of Theorem \ref{theo-typeD}(Case 3)}]
First, we look at the vertex $(0, 2)$. $\omega(0, 2)=(m,2)$ and $r\in\SE_G(2)$ if and only if $\rem{rm}_n<m$. If $m\mid n$, equivalently $3\mid u$, then $\rem{rm}_n<m$ happens only for $r=i\Fb_1$. In this case $\rem{\Fb_1m}_n=0<m$ and thus $\rd_G(2)=\Fb_1-1$. If  {$3\nmid u$}, then $m\nmid n$. By Lemma \ref{lem-rem<m}, $\rem{rm}_n\geq m$ for {$0<r\leq\Fb_1$} and $\rem{rm}_n<m$ for $r=\Fb_1+1$. Hence $\rd_G(2)=\Fb_1$ in this case.

For the vertex $(0,1)$, {$\omega(0,1)=(m,1)$}, and $r\in \SE_G(1)$ if and only if there is some integer $k$ such that
$$(\tau^n\phi)^k\omega^r(0,1)\in H^-(0, 1),$$
equivalently, $\rem{rm}_n<m$  and either of the following conditions holds

\medskip
(a) $\rem{rm}_n=1$ and $\sigma^k(1)\neq 1$ ($3\nmid k$);

(b) $\rem{rm}_n\neq 1$ and $\sigma^k(1)= 1$ ($3\mid k$),

\medskip
\noindent
where $k=(rm-\rem{rm}_n)/n$. For convenience, we write $T_r:=(rm-\rem{rm}_n)/n$.

If $3\mid u$, then $\rem{rm}_n<m$ if and only if $r=i\Fb_1$. In this case $\rem{i\Fb_1m}_n=0$ and $T_{i\Fb_1}=i$. Hence $3\Fb_1\in\SE_G(1)$ and $r\notin\SE_G(1)$ for all $r<3\Fb_1$, that is, $\rd_G(1)=3\Fb_1-1$.

Assume that $3\nmid u$. Then, for $r<\Fb_1+\Fb_2$,  $\rem{rm}_n<m$ if and only if
$r=p\Fb_1+1$ with $1\leq p\leq k_2$. In this case $\rem{rm}_n=m-ps_2<m$ and $T_r=(rm -\rem{rm}_n)/n=p$.

If $u\equiv 1\pmod{3}$, then $s_2=2$. It follows that $\rem{(\Fb_1+1)m}_n=m-s_2=1$ and $T_{\Fb_1+1}=1$ is not divided by $3$. Hence $\Fb_1+1\in\SE_G(1)$ and therefore $\rd_G(1)=\Fb_1$.

If $u\equiv 2\pmod{3}$, then $k_2=3$ and $s_2=1$. $\rem{(\Fb_1+1)m}_n=m-s_2=2$. However $T_{\Fb_1+1}=1$ is not divided by $3$. Thus $\Fb_1+1\notin\SE_G(1)$. Finally, $\rem{(2\Fb_1+1)m}_n=m-2s_2=1$ and $T_{2\Fb_1+1}=2$ is not divided  by $3$. This implies that $2\Fb_1+1\in\SE_G(1)$. Hence $\rd_G(1)=2\Fb_1$ in this case.

Finally, $\rd_G(m_\pm)=\rd_G(1)$ by symmetry. This finishes the proof.
\end{proof}

\section{Rigidity degrees of indecomposable modules: type $E$}
In this section, we assume that $\Lambda$ is of type $E$. The Dynkin graph of type $E$ is labelled as follows.
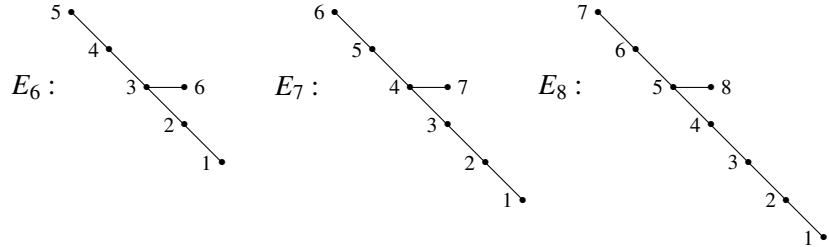
\begin{center}
	\begin{tikzpicture}
		\begin{scope}[scale=0.5]
			\node at (-1,2) {$E_6:$};
			\draw (0,4) node {${\scriptscriptstyle \bullet}$} node[left] {${\scriptstyle 5}$}--(1,3) node {${\scriptscriptstyle \bullet}$} node[left] {${\scriptstyle 4}$}
			--(2,2)
			node {${\scriptscriptstyle \bullet}$} node[left] {${\scriptstyle 3}$}
			--(3,1)
			node {${\scriptscriptstyle \bullet}$} node[left] {${\scriptstyle 2}$}
			--(4,0) node {${\scriptscriptstyle \bullet}$} node[left] {${\scriptstyle 1}$};
	\draw (2,2)--(3,2) node {${\scriptscriptstyle \bullet}$} node[right] {${\scriptstyle 6}$};
		\end{scope}

			\begin{scope}[scale=0.5,xshift=7cm]
				\node at (-1,2) {$E_7:$};
				\draw (0,4) node {${\scriptscriptstyle \bullet}$} node[left] {${\scriptstyle 6}$}--(1,3) node {${\scriptscriptstyle \bullet}$} node[left] {${\scriptstyle 5}$}
				--(2,2)
				node {${\scriptscriptstyle \bullet}$} node[left] {${\scriptstyle 4}$}
				--(3,1)
				node {${\scriptscriptstyle \bullet}$} node[left] {${\scriptstyle 3}$}
				--(4,0) node {${\scriptscriptstyle \bullet}$} node[left] {${\scriptstyle 2}$}
				--(5,-1) node {${\scriptscriptstyle \bullet}$} node[left] {${\scriptstyle 1}$};
		\draw (2,2)--(3,2) node {${\scriptscriptstyle \bullet}$} node[right] {${\scriptstyle 7}$};
			\end{scope}

			\begin{scope}[scale=0.5,xshift=14cm]
				\node at (-1,2) {$E_8:$};
				\draw (0,4) node {${\scriptscriptstyle \bullet}$} node[left] {${\scriptstyle 7}$}--(1,3) node {${\scriptscriptstyle \bullet}$} node[left] {${\scriptstyle 6}$}
				--(2,2)
				node {${\scriptscriptstyle \bullet}$} node[left] {${\scriptstyle 5}$}
				--(3,1)
				node {${\scriptscriptstyle \bullet}$} node[left] {${\scriptstyle 4}$}
				--(4,0) node {${\scriptscriptstyle \bullet}$} node[left] {${\scriptstyle 3}$}
				--(5,-1) node {${\scriptscriptstyle \bullet}$} node[left] {${\scriptstyle 2}$}
				-- (6,-2) node {${\scriptscriptstyle \bullet}$} node[left] {${\scriptstyle 1}$};
		\draw (2,2)--(3,2) node {${\scriptscriptstyle \bullet}$} node[right] {${\scriptstyle 8}$};
			\end{scope}
	\end{tikzpicture}
\end{center}
The type of $\Lambda$ is one of the following.
\begin{itemize}
	\item $(E_r, u, 1)$, $r=6,7,8$,  $u\in\mathbb{N}$;
	\item $(E_6, u, 2)$, $u\in\mathbb{N}$.
\end{itemize}
The stable Auslander-Reiten quiver is $\mathbb{Z}E_r/G$, where $G$ is a cyclic group generated by $\tau^n$ for type $(E_r, u, 1)$ where $n=um_{E_r}$. Here $m_{E_6}=11, m_{E_7}=17,m_{E_8}=29$. For type $(E_6, u, 2)$, the group $G$ is generated by $\tau^{n-6}\omega$. Let $h_{\Delta}=m_{\Delta}+1$ be the Coxeter number, and let $h_{\Delta}^*=h_{\Delta}/2$ be half of it.

The main result of this section is the following theorem.
\begin{Theo}\label{theo-typeE}
	Suppose that $\Lambda$ is of type $(\Delta, u, s)$, $\Delta=E_r, r=6,7,8$. Let $n=um_{\Delta}$ and $\kk=\wseq{h^*_{\Delta},n}$ be the weight sequence. Suppose that $\Fb_i:=\Fb_i(\kk)$ is the weighted Fibonacci sequence. Let $X$ be an indecomposable $\Lambda$-module corresponding to the vertex $t$ on $\Delta$. Then $\rd(X)$ can be read from Table \ref{tab-typeE}.

	\renewcommand{\arraystretch}{1.3}
\begin{table}[h!]
		\centering
		
		$\Delta=E_7,\quad  h^*_{\Delta}=9$

		\medskip
	\begin{tabular}{l|l|l|l|l|l|l|l|l|l}
	\hline
	$\rem{u}_9$ & 0 & 1& 2&3 &4 &5 & 6 & 7 & 8\\
	\hline
	$t=1$ & $\Fb_1-1$ & $\Fb_1+3\Fb_2$ & $\Fb_1+\Fb_2$ & $\Fb_3-1$ & $\Fb_1$ & $\Fb_1+3\Fb_2$ & $\Fb_2-1$ & $\Fb_1+\Fb_2$ & $\Fb_1$ \\
	\hline
	$t=2$ & $\Fb_1-1$ & $\Fb_1$ & $\Fb_1+\Fb_2$ & $\Fb_1$ & $\Fb_1$ & $\Fb_1$ & $2\Fb_1$ & $\Fb_1$ & $\Fb_1$ \\
	\hline
	$t=3,4,5$ & $\Fb_1-1$ & $\Fb_1$ & $\Fb_1$ & $\Fb_1$ & $\Fb_1$ & $\Fb_1$ & $\Fb_1$ & $\Fb_1$ & $\Fb_1$ \\
	\hline
	$t=6$ & $\Fb_1-1$ & $\Fb_1+2\Fb_2$ & $\Fb_1+3\Fb_2$ & $\Fb_1$ & $\Fb_1+\Fb_2$ & $\Fb_1$ & $2\Fb_1$ & $2\Fb_1$ & $\Fb_1$ \\
	\hline
	$t=7$ & $\Fb_1-1$ & $\Fb_1+\Fb_2$ & $\Fb_1$ & $\Fb_1$ & $\Fb_1$ & $\Fb_1$ & $\Fb_1$ & $2\Fb_1$ & $\Fb_1$ \\
	\hline
	\end{tabular}

       \medskip
	$\Delta=E_8, \quad h^*_{\Delta}=15$

	\medskip
	\begin{tabular}{l|l|l|l|l|l|l|l|l}
	\hline
	$\rem{u}_{15}$ & 0 & 1& 2&3 &4 &5 & 6 & 7 \\
	\hline
	$t=1$ & $\Fb_1-1$ & $\Fb_1+4\Fb_2$ & $\Fb_3$ & $\Fb_1+2\Fb_2$ & $\Fb_1+\Fb_2+\Fb_3$ & $\Fb_1$ & $2\Fb_3$ & $\Fb_1+\Fb_2$  \\
	\hline
	$t=2$ & $\Fb_1-1$ & $\Fb_1$ & $\Fb_1+\Fb_2$ & $\Fb_1+\Fb_2$ & $\Fb_1$ & $\Fb_1$ & $\Fb_1$ & $\Fb_1+\Fb_2$  \\
	\hline
	$t=3,4,5,6$ & $\Fb_1-1$ & $\Fb_1$ & $\Fb_1$ & $\Fb_1$ & $\Fb_1$ & $\Fb_1$ & $\Fb_1$ & $\Fb_1$  \\
	\hline
	$t=7$ & $\Fb_1-1$ & $\Fb_1+2\Fb_2$ & $\Fb_1+2\Fb_2$ & $\Fb_1$ & $\Fb_1+\Fb_2$ & $\Fb_1$ & $\Fb_1$ & $\Fb_1+\Fb_2$  \\
	\hline
	$t=8$ & $\Fb_1-1$ & $\Fb_1+\Fb_2$ & $\Fb_1$ & $\Fb_1$ & $\Fb_1$ & $\Fb_1$ & $\Fb_1$ & $\Fb_1$  \\
	\hline
	\end{tabular}

	\medskip
	\begin{tabular}{l|l|l|l|l|l|l|l}
	\hline
	$\rem{u}_{15}$ & 8 & 9& 10&11&12 &13 & 14 \\
	\hline
	$t=1$ & $\Fb_1+\Fb_2$ &  $\Fb_1$ &$2\Fb_1$ & $\Fb_1+\Fb_2$ & $2\Fb_1$ & $3\Fb_1$ & $\Fb_1$ \\
	\hline
	$t=2$ & $\Fb_1$ & $\Fb_1$ & $\Fb_1$ & $\Fb_1+\Fb_2$ & $2\Fb_1$ & $\Fb_1$ & $\Fb_1$ \\
	\hline
	$t=3,4,5,6$ & $\Fb_1$ &  $\Fb_1$ & $\Fb_1$ & $\Fb_1$ & $\Fb_1$ & $\Fb_1$ & $\Fb_1$  \\
	\hline
	$t=7$ & $\Fb_1$ &  $\Fb_1$ & $2\Fb_1$ & $\Fb_1$ & $2\Fb_1$ & $2\Fb_1$ & $\Fb_1$ \\
	\hline
	$t=8$ & $\Fb_1$ & $\Fb_1$ & $\Fb_1$ & $\Fb_1$ & $\Fb_1$ & $2\Fb_1$ & $\Fb_1$  \\
	\hline
	\end{tabular}

   $(E_6, u, s), s=1, 2$

	\begin{tabular}{l|l|l|l|l|l|l}
		\hline
		$\rem{u}_{6}$ & 0 & 1& 2&3 &4 &5\\
		\hline
		$t=3$ & $\Fb_1-1$ & $\Fb_1$ & $\Fb_1$ & $\Fb_1$ & $\Fb_1$ & $\Fb_1$  \\
		\hline
		$t=6$ & $\Fb_1-1$ & $\Fb_1+\Fb_2$ & $\Fb_1$ & $\Fb_1$ & $2\Fb_1$ & $\Fb_1$  \\
		\hline
		\end{tabular}

		\medskip
	   $(E_6, u, 1), \lfloor u/6\rfloor\mbox{ is even}$
	
	   $(E_6, u, 2), \lfloor u/6\rfloor\mbox{ is odd}$

	   \begin{tabular}{l|l|l|l|l|l|l}
		\hline
		$\rem{u}_{6}$ & 0 & 1& 2&3 &4 &5\\
		\hline
		$t=1,5$ & $\Fb_1-1$ & $\Fb_1+2\Fb_2$ & $\Fb_1+\Fb_3$ & $\Fb_1$ & $2\Fb_1$ & $3\Fb_1$  \\
		\hline
		$t=2,4$ & $\Fb_1-1$ & $\Fb_1$ & $\Fb_1$ & $\Fb_1$ & $\Fb_1$ & $2\Fb_1$  \\
		\hline
		\end{tabular}

		\medskip
	   $(E_6, u, 1), \lfloor u/6\rfloor\mbox{ is odd}$

	   $(E_6, u, 2),  \lfloor u/6\rfloor\mbox{ is even}$

	   \begin{tabular}{l|l|l|l|l|l|l}
		\hline
		$\rem{u}_{6}$ & 0 & 1& 2&3 &4 &5\\
		\hline
		$t=1,5$ & $2\Fb_1-1$ & $\Fb_1+4\Fb_2$ & $\Fb_1$ & $\Fb_1+\Fb_2$ & $2\Fb_1$ & $\Fb_1$  \\
		\hline
		$t=2,4$ & $2\Fb_1-1$ & $\Fb_1$ & $\Fb_1$ & $\Fb_1$ & $\Fb_1$ & $\Fb_1$  \\
		\hline
		\end{tabular}
	\caption{\label{tab-typeE} Rigidity degrees: Type $E$.}
\end{table}
\end{Theo}

Let us explain how we calculate the rigidity degrees of indecomposable modules in type $E$.

\medskip
We keep the notations in Theorem \ref{theo-typeE}. For simplicity, we assume that $|\kk|=d+1$ and let $s_{-1}, s_0, \cdots, s_{d+1}$ be the remainder sequence. Note that $h_{\Delta}^*<n$ always holds. Hence $s_1=h_{\Delta}^*$ and $s_2<h_{\Delta}^*$. Each positive integer $r$ can be written as
$$r=a\Fb_{d+1}+\sum_{i=1}^{d+1}\lambda_i\Fb_{i-1}\quad (\star)$$ with $a\geq 0$, $0\leq \lambda_i\leq k_i$ for all $i$ and $\lambda_1>0$. It follows from Lemma \ref{lemma-rs1-rem} and the proof of Lemma \ref{lemma-rem-range} that
$$rh^*_{\Delta}\equiv \sum_{i=1}^{d+1}(-1)^{i-1}\lambda_is_i\pmod{n} \mbox{  and } (\lambda_1-1)h^*_{\Delta}\leq \sum_{i=1}^{d+1}(-1)^{i-1}\lambda_is_i\leq n.$$
%The right equality holds if and only if $\sum_{i=1}^{d+1}\lambda_i\Fb_{i-1}=\Fb_{d+1}$. In this case $\lambda_1=1$, $\lambda_i=k_i$ when $i$ is even and zero when $i$ is odd.
It follows that $\rem{rh^*_{\Delta}}_n<h^*_{\Delta}$ only if $\lambda_1=1$, that is, $r=a\Fb_{d+1}+1+\sum_{i=2}^{d+1}\lambda_i\Fb_{i-1}$. This will be frequently used below.

Assume that $\Delta=E_7$ or $E_8$ and  define
$$\mathcal{X}_{\Delta}(t):=\{0\leq x< h^*_{\Delta}\mid (x,t)\in H^-(0,t)\}.$$
One can draw the picture of $H^-(0,t)$ by the algorithm given in \cite[4.4.2]{Iyama2005c} and write down the elements in $\mathcal{X}_{\Delta}(t)$ explicitly. Note that $\omega(x,t)=(x+h_{\Delta}^*,t)$. Since $G$ is generated by $\tau^n$ in this case, it is easy to see that a positive integer $r$ belongs to $\SE_G(t)$ if and only if $\rem{rh^*_{\Delta}}_n\in\mathcal{X}_{\Delta}(t)$. Clearly $\Fb_{d+1}\in\SE_G(t)$. Suppose $r\leq \Fb_{d+1}$ belongs to $\SE_G(t)$. We write $r$ as the form $(\star)$ with $a=0$.  Since all elements in $\mathcal{X}_{\Delta}(t)$ is less than $h_{\Delta}^*$, $r$ must be of the form  $1+\sum_{i=2}^{d+1}\lambda_i\Fb_{i-1}$.  The possibilities of $\lambda_i, i\geq 2$ are very limited. One can carefully check for which $\lambda_i, i\geq 2$, the remainder
$$h^*_{\Delta}+\sum_{i=2}^{d+1}(-1)^{i-1}\lambda_is_i.$$
is in $\mathcal{X}_{\Delta}(t)$, choose the smallest
$r=1+\sum_{i=2}^{d+1}\lambda_i\Fb_{i-1}$, and get the rigidity degree
$\rd_G(t)=\sum_{i=2}^{d+1}\lambda_i\Fb_{i-1}$.

Assume that $\Delta=E_6$. In this case $h^*_{\Delta}=6$ and $G$ is generated by $\tau^n\phi$, where $\phi$ is the identity when $s=1$ and $\phi=\tau^{-6}\omega$ when $s=2$. The automorphism $\omega$ satisfies the following conditions
$${\omega(x,t)=(x+t+3,6-t)}, t\leq 5, \quad \omega(x, 6)=(x+6,6), \quad \omega^2=\tau^{2h^*_{\Delta}}.$$
For each vertex $t$ of $\Delta$, one can draw the picture of $H^-(0,t)$, and find
$$\mathcal{X}_0(t)=\{0\leq x<h_{\Delta}^*\mid (x,t)\in H^-(0,t)\},$$
$$\mathcal{X}_1(t)=\{0\leq x<h_{\Delta}^*\mid \omega(x-h^*_{\Delta},t)\in H^-(0,t)\}.$$
\iffalse
It is straightforward to check that
$$\mathcal{X}_0(1)=\{0,3\}, \quad \mathcal{X}_1(1)=\{2,5\},$$
$$\mathcal{X}_0(2)=\{0,1,2,3,4\}, \quad \mathcal{X}_1(2)=\{1,2,3,4,5\},$$
$$\mathcal{X}_0(3)=\{0,1,2,3,4,5\}, \quad \mathcal{X}_0(6)=\{0,2,3,5\}.$$
\fi
To determine $\rd_G(t)$, we need the sequences $k_i, i\geq 2$, $s_i, i\geq 1$, and the parity of $\Fb_1$ which is { opposite to the parity of $\lfloor u/6\rfloor$ when $6\nmid u$
 and is the same as the parity of $\lfloor u/6\rfloor$ when $6\mid u$}. The sequence $k_i,i\geq 2$ is completely determined by $s_2$ which can only be a non-negative integer less than $6$. For each $t$, $r\in\SE_G(t)$ if and only if there is some integer $k$ such that $$(\tau^n\phi)^k\omega^r(0,t)\in H^-(0,t).$$
If $t=3,6$, then the above condition is equivalent to $\rem{rh^*_{\Delta}}_n\in\mathcal{X}_0(t)$, since $\phi(x,t)=(x,t)$ always holds for $t=3,6$. Then the method for $E_7$ and $E_8$ also applies.

Suppose that $t=1, 2$ and $s=1$. Then $r=2l\in\SE_G(t)$ if and only if there is an integer $k$ such that
$$\omega^{2l}(\tau^n)^{k}(0,t)\in H^-(0,t)$$
Equivalently, $(2lh^*_{\Delta}+kn,t)\in H^-(t)$. That is $\rem{rh^*_{\Delta}}_n\in\mathcal{X}_0(t)$. An odd integer $r=2l+1\in\SE_G(t)$ if and only if there is an integer $k$ such that $\omega^{2l+1}\tau^{kn}(0,t)\in H^-(0,t)$, equivalently $\omega((2l+1)h^*_{\Delta}-h^*_{\Delta}+kn,t)\in H^-(0,t)$, which is further equivalent to $\rem{rh^*_{\Delta}}_n\in\mathcal{X}_1(t)$. In either case, we have $\rem{rh^*_{\Delta}}_n<h^*_{\Delta}$. Each $r$ with $\rem{rh_{\Delta}^*}_n<h_{\Delta}^*$ is of the form
$$a\Fb_{d+1}+1+\sum_{i=2}^{d+1} \lambda_i\Fb_{i-1}, \quad a\geq 0, 0\leq \lambda_i\leq k_i.$$
Its parity can be deduced from that of $\Fb_1$ and the coefficients $a, \lambda_i, i\geq 2$ using the relation $\Fb_{i}=k_i\Fb_{i-1}+\Fb_{i-2}$. The remainder
$$\rem{rh_{\Delta}^*}_n=h_{\Delta}^*+\sum_{i=2}^{d+1}(-1)^{i-1}\lambda_is_i$$
can be calculated, and one can check whether it is in $\mathcal{X}_0(t)$ when $r$ is even, or in $\mathcal{X}_1(t)$ when $r$ is odd.

It remains to consider the case $s=2$ and $t=1,2$. In this case $r\in\SE_G(t)$ if and only if there is an integer $k$ such that
$$\omega^{r}(\tau^{n-6}\omega)^{2k}(0,t)\in H^-(0,t)\mbox{ or }\omega^{r}(\tau^{n-6}\omega)^{2k+1}(0,t)\in H^-(0,t)$$
If $r=2l$ is even, then this is equivalent to
$$\rem{rh^*_{\Delta}}_n\in \mathcal{X}_0(t) \mbox{ and } (rh^*_{\Delta}-\rem{rh^*_{\Delta}}_n)/n \mbox{ is even}, or $$
$$\rem{rh^*_{\Delta}}_n\in \mathcal{X}_1(t) \mbox{ and } (rh^*_{\Delta}-\rem{rh^*_{\Delta}}_n)/n \mbox{ is odd}.$$
Similarly, $r=2l+1\in\SE_G(t)$ if and only if
$$\rem{rh^*_{\Delta}}_n\in \mathcal{X}_1(t) \mbox{ and } (rh^*_{\Delta}-\rem{rh^*_{\Delta}}_n)/n \mbox{ is even, or} $$
$$\rem{rh^*_{\Delta}}_n\in \mathcal{X}_0(t) \mbox{ and } (rh^*_{\Delta}-\rem{rh^*_{\Delta}}_n)/n \mbox{ is odd}.$$
Again, in all cases, we need $\rem{rh^*_{\Delta}}_n<h^*_{\Delta}$. This implies that
$$r=a\Fb_{d+1}+1+\sum_{i=2}^{d+1}\lambda_i\Fb_{i-1},\quad a\geq 0,\quad 0\leq \lambda_i\leq k_i.$$
The remainder
$$\rem{rh^*_{\Delta}}_n=h^*_{\Delta}+\sum_{i=2}^{d+1}(-1)^{i-1}s_i$$
only depends on $\lambda_i\leq k_i, s_i, i\geq 2$, which has very limited choices. Again the parity of $\Fb_1$ and $a, \lambda_i, i\geq 2$ determine the parity of $r$.  By the proof of Lemma \ref{lemma-Bls1-rem}, we have
$$\Fb_{i-1}h^*_{\Delta}-(-1)^{i-1}s_i=\Fb_{i-1}(\kk')n,$$
where $\kk'$ is the sequence $k_2, k_3, \cdots$. Thus, $(rh^*_{\Delta}-\rem{rh^*_{\Delta}}_n)/n$ has the same parity as
$$a\Fb_{d+1}(\kk')+\sum_{i=2}^{d+1} \lambda_i\Fb_{i-1}(\kk').$$
which is easily deduced from $k_i, i\geq 2$ and the coefficients $a$,  $\lambda_i, i\geq 2$.

Altogether, given the parity of $\lfloor u/6\rfloor$ and $0\leq s_2<6$, one can come up with a computer algorithm to determine the coefficients $\lambda_i$ and $a$ so that $r$ is the smallest positive integer belonging to $\SE_G(t)$.

%%%%%%%%%%%%%%%%%%%%%%%%%%%%%%%%
\section{Rigidity  dimension via maximal orthogonal modules}
Let us recall some basic facts on maximal orthogonal modules from \cite{Iyama2005c}. Let $\Lambda$ be an algebra, and let $M$ be a $\Lambda$-module. For a non-negative integer $r$, define
$$M^{\perp_r}:=\{Y\in\modcat{\Lambda}\mid \Ext_{\Lambda}^i(M, Y)=0\mbox{ for all }0<i\leq r\}.$$
One can similarly define ${}^{\perp_r}M$. $M$ is called a maximal
$r$-orthogonal module if
$$M^{\perp_r}=\add(M)={}^{\perp_r}M. $$
In case that $\Lambda$ is self-injective, this is equivalent to $M^{\perp_r}=\add(M)$, or equivalently  $\add(M)={}^{\perp_r}M$. The endomorphism algebra of a maximal $r$-orthogonal module is called an $(r+2)$-Auslander algebra which has global dimension at most $r+2$ and dominant dimension at least $r+2$. Particularly $\rd(M)\geq r$.

For a Dynkin quiver $\Delta$, a subset $M$ of vertices on $\mathbb{Z}\Delta$ is called a maximal $r$-orthogonal subset if
$$\mathbb{Z}\Delta\backslash M=\bigcup_{v\in M, 0<i\leq r} H^+(\omega^iv). $$
Note that maximal $r$-orthogonal subset is always $\tau\omega^r$-stable (\cite[Proposition 4.2.1]{Iyama2005c}).
Suppose that $\Lambda$ is a representation-finite self-injective algebra such that its stable AR-quiver $\Gamma_s(\Lambda)$ is isomorphic to $\mathbb{Z}\Delta/G$. Let $\pi: \mathbb{Z}\Delta\lra \Gamma_s(\Lambda)$ be the canonical map. It was proved in \cite[Theorem 4.2.2]{Iyama2005c} that a $\Lambda$-module $M$ is maximal $r$-orthogonal if and only if the set of preimages of the indecomposable non-projective direct summands of $M$ under $\pi$ is a maximal $r$-orthogonal subset of $\mathbb{Z}\Delta$.

Based on the formulae of rigidity degrees of indecomposable modules, it is possible to determine the rigidity dimension of some representation-finite self-injective algebras. The idea is as follows. For an indecomposable non-semisimple representation-finite self-injective algebra $\Lambda$, take an indecomposable module $X$ with the maximal rigidity degree $r$. If we are so lucky that $\Lambda\oplus X$ is a maximal $r$-orthogonal module, then $\gldim \End_{\Lambda}(\Lambda\oplus X)$ is finite and thus $\rigdim \Lambda\geq r+2$. By the maximality of $r$, every non-projective generator-cogenerator $M$ has rigidity degree at most $r$. It follows that $\rigdim \Lambda=r+2$.

\begin{Theo}\label{theo-max-orth-typeA}
Let $\Lambda$ be an indecomposable representation-finite non-semisimple self-injective algebra of type $(A_{m-1}, n/(m-1), 1)$.  Suppose that $X$ is an indecomposable $\Lambda$-module  corresponding to the vertex $(x,1)$ with $\rd(X)=r$. Then $\Lambda\oplus X$ is maximal $r$-orthogonal if and only if one of the following conditions holds.

\smallskip
$(1)$ $m=2$ and $n=2a$, $a\in\mathbb{N}$.  In this case, $r=2a-1$ and $\rigdim\Lambda=2a+1$.

%$(2)$ $m=2$ and $n=2a+1, a\in\mathbb{N}$, and $r=2a$.

$(2)$ $n=am-1$, $a\in\mathbb{N}$. In this case, $r=2(am-a-1)$ and  $\rigdim \Lambda=2(am-a)$.
\end{Theo}

\begin{proof}
	Let $\kk=\wseq{m,n}$ be the weight sequence of $m,n$. Note that $m\geq 2$ since $\Lambda$ is not semisimple. For simplicity, we assume that $|\kk|=d+1$ and let $\Fb_i:=\Fb_i(\kk), i=1, 2,\cdots, d+1$ be the corresponding weighted Fibonacci sequence.

 The stable AR-quiver of $\Lambda$ is of the form $\mathbb{Z}A_{m-1}/G$, where $G=\langle\tau^n\rangle$. Without loss of generality, one can assume that $X$ corresponds to the vertex $(0,1)$ on $\mathbb{Z}A_{m-1}$. By \cite[Theorem 4.2.2]{Iyama2005c}, $\Lambda\oplus X$ is maximal $r$-orthogonal if and only if the orbit $G(0,1)$ is a maximal $r$-orthogonal subset of $\mathbb{Z}A_{m-1}$, that is, $G(0,1)$ satisfies the following condition:
$$\mathbb{Z}A_{m-1}\backslash G(0,1)=\bigcup_{v\in G(0,1), 0<i\leq r}H^+(\omega^iv) \quad (\dagger)$$
%This is equivalent to
%\medskip
%For each vertex $p$ not in $G(0,1)$, there is $0<i\leq r$ such that $Gp\cap H^+(\omega^i(0,t))\neq\emptyset$.
% \medskip
%\noindent
That is, for each vertex $(x,t)$ in
$$S:=\{(x,t)\mid
		0\leq x<n, 1\leq t\leq m-1
	\}\backslash\{(0,1)\},$$
there is a positive  integer $i\leq r$ and an integer $a$ such that $(x+an,t)\in H^+(\omega^i(0,1))$. If $i=2b$ is even, this is equivalent to $x+an=bm-t+1$. If $i=2b-1$ is odd, this is equivalent to $x+an=(b-1)m+1$.

\medskip
Assume that $r=2l$ is even. Then $\omega^r=\tau^{lm}$.  Since $G(0,1)$ is $\tau\omega^r$-stable, there is some integer $a$ such that $\tau\omega^r(0,1)=(an,1)$. This implies that $\rem{lm}_n=n-1$.  By checking $\rd_G(1)$ given in Theorem \ref{theo-typeA}, this happens if and only if

\medskip
\begin{itemize}
	\item $r=2\Fb_{d}$, $d$ is odd and $s_{d+1}=1$, or
	\item $r=2(\Fb_{d+1}-\Fb_{d})$, $d$ is even and $s_{d+1}=1$.
\end{itemize}

\medskip\noindent
By the discussion above, condition ($\dagger$) is equivalent to that, for each $(x,t)\in S$, there is $0<b\leq l$ such that $\rem{bm}_n\equiv x-1+t\pmod{n}$ or $\rem{(b-1)m}_n\equiv x-1\pmod{n}$. We denote this condition by ($\ddagger$).

\iffalse
\medskip
\begin{itemize}
	\item[(i)] For each $0<x<n$ and $1\leq t\leq m-1$, there is a positive integer $b\leq l$ such that $bm\equiv x+t-1\pmod{n}$ or $(b-1)m\equiv x-1\pmod{n}$;
	\item[(ii)] For each $2\leq t\leq m-1$, there is positive  integer $b\leq l$ such that $bm\equiv t-1\pmod{n}$ or $(b-1)m\equiv -1\pmod{n}$.
\end{itemize}
\fi

\medskip
If $d=-1$, then $r=2\Fb_{d}=0$,  $m=k_0n$ for some positive integer $k_0$ and $n=s_{d+1}=1$. If $m>2$,  then $(0,2)\in S$. There is no positive integer $0<b\leq \Fb_d=0$ satisfying the above condition, and $\Lambda\oplus X$ is not maximal $r$-orthogonal. If $m=2$, then $X$ is the only indecomposable non-projective $\Lambda$-module. Hence $\Lambda\oplus X$ is maximal $0$-orthogonal and $\rigdim\Lambda=2$. Taking $a=1$, it is easy to check that $n=am-1$ and $\rigdim\Lambda=2(am-a)$.

Now assume that $d\geq 0$. Then $n=s_0\geq s_d>s_{d+1}=1$ and $k_{d+1}=s_d/s_{d+1}>1$. Hence
$$\Fb_{d+1}=k_{d+1}\Fb_{d}+\Fb_{d-1}\geq 2\Fb_{d}>\Fb_d.$$
It follows that $r=2l>0$ with $0<l<\Fb_{d+1}$ in both cases above. Set $$\mathcal{X}=\{\rem{bm}_n\mid 1\leq b\leq l\}.$$
Since $l<\Fb_{d+1}$, the remainder $\rem{bm}_n, 1\leq b\leq l$ are pairwise distinct. Hence $l=|\mathcal{X}|$.  Since  $\rem{lm}_n=n-1$,  there cannot be any integer $0<b\leq l$ such that $(b-1)m\equiv -1\pmod{n}$. For the vertices $(0,t), t=2,\cdots, m-1$, the condition ($\ddagger$) holds if and only if   $\{\rem{b}_n\mid 1\leq b \leq m-2\}\subseteq \mathcal{X}$.
%For $x=1$, taking $b=1$, we have $(b-1)m\equiv 0\equiv x-1\pmod{n}$ and thus (i) is true for $x=1$.
For vertices $(x,t)$ with $0<x<n$, the condition ($\ddagger$) means that either $x-1$ or $\rem{x+t-1}_n$ belongs to $\mathcal{X}$.  Altogether $\Lambda\oplus X$ is maximal $2l$-orthogonal if and only if the following  conditions hold.

\medskip
(a) $\{\rem{b}_n\mid 1\leq b\leq m-2\}\subseteq \mathcal{X}$;

(b) For each $0<x<n$, either $x-1$ or $\rem{x-1+t}_n, t=1,\cdots,m-1$ belongs to $\mathcal{X}$.

\medskip
Assume that  $m>n$. If $m\geq n+2$, then $n\in\{1,\cdots, m-2\}$ and consequently $\rem{n}_n=0\notin\mathcal{X}$, and (a) is not satisfied. In case that $m=n+1$, one has $d=0$ and $r=2(\Fb_1-\Fb_0)=2(n-1)$,  $\mathcal{X}=\{1,2,\cdots, m-2\}$. Conditions (a) and (b) are both satisfied. $\Lambda\oplus X$ is a maximal $2(m-2)$-orthogonal module and $\rigdim \Lambda=2(m-1)$.

$m=n$ cannot happen since $d\geq 0$.

Now assume that $m<n$. Then $s_1=m$ and $n=k_1m+s_2$. The conditions (a) and (b) above imply that there are at most $k_1$ positive integers less than $n$ not in $\mathcal{X}$, that is, $l=|\mathcal{X}|\geq n-1-k_1$.  Moreover,
$$l\geq n-1-k_1\geq \Fb_{d+1}-(\Fb_1+\Fb_0)\geq \Fb_{d+1}-\Fb_2.$$
If $l=\Fb_d$, then $d$ is odd and $\Fb_d\leq \Fb_{d+1}-\Fb_d\leq \Fb_2$. Since $d\geq 0$, we have $d=1$, $s_2=s_{d+1}=1$. Thus $n=k_1m+1$. The inequality above then implies that $\Fb_1\geq \Fb_2-(\Fb_1+\Fb_0)$, that is, $\Fb_2\leq 2\Fb_1+\Fb_0$. Hence $m=k_2\leq 2$. This forces $m=2$ and $n=2k_1+1=2(k_1+1)-1$. Now $\mathcal{X}=\{n-1, n-3, \cdots, 2\}$. It is straightforward to check that the conditions (a) and (b) hold. Hence $\Lambda\oplus X$ is a maximal $2k_1$-orthogonal module. Taking $a=k_1+1$, one has $\rigdim\Lambda=2k_1+2=2(am-a)$.

If $l=\Fb_{d+1}-\Fb_d$, then $d$ is even and the inequality above implies that $\Fb_d\leq \Fb_2$. Therefore $d=0$ or $2$. $d=0$ cannot happen, otherwise, $m=s_1=s_{d+1}=1$ which is impossible. Hence $d=2$. Again the inequality provides
$$\Fb_3-\Fb_2\geq \Fb_3-(\Fb_1+\Fb_0),$$
equivalently $\Fb_2\leq \Fb_1+\Fb_0$. This forces $k_2=1$ and thus $n=k_1m+s_2$ and $m=s_2+1$. That is, $n=k_1m+(m-1)$. It is straightforward to check that $\Fb_3=n$,
$$\mathcal{X}=\{1,\cdots,n-1\}\backslash\{n-m, \cdots, n-k_1m\}.$$
The conditions (a) and (b) are both satisfied. Taking $a=k_1+1$, one has $\Fb_3-\Fb_2=n-(k_1+1)=2(am-a-1)$. Hence  $\Lambda\oplus X$ is a maximal $2(am-a-1)$-orthogonal module and  $\rigdim\Lambda=2(am-a)$.

\medskip

If $r=2l+1$ is odd, then
$$\tau\omega^r(0,1)=(lm+2,m-1)=(an,1)$$
if and only if $m=2$ and $lm+2=an$. By checking $\rd_G(1)$ in Theorem \ref{theo-typeA}, this happens if and only if $s_{d+1}\geq 2$. In this case $r=2\Fb_{d+1}-1$. If $d\geq 0$, then $2\leq s_{d+1}\leq s_1\leq m=2$. It follows that $d=0$, $n>m$ and $n=k_1m$ for some $k_1>1$. If $d=-1$, then $2\leq s_0=n\leq m=2$. Altogether We have $m=2$ and $n=2a$ for some positive integer $a$. In this case $\rd(X)=2a-1$, and $H^+(\omega^i(0,1))=\{(i,1)\}$ for all $i\leq 2a-1$. It follows easily that $A\oplus X$ is a maximal $(2a-1)$-orthogonal module and $\rigdim \Lambda=2a+1$. \end{proof}

\begin{Theo}\label{theo-max-orth-typeA2}
	Let $\Lambda$ be an indecomposable representation-finite non-semisimple self-injective algebra of type $(A_{m-1}, u, 2)$, and let $n=u(m-1)-m/2$. Suppose that $X$ is an indecomposable $\Lambda$-module  corresponding to the vertex $(x,1)$ with $\rd(X)=r$. Then $\Lambda\oplus X$ is maximal $r$-orthogonal if and only if  $n=am-1$ for some integer $a>1$, $r=2am+m-2a-3$. In this case $\rigdim\Lambda=(2a+1)(m-1)$.
\end{Theo}
\begin{proof}
Set $M=m+n$ and $N=m+2n$. Let $\kk=\wseq{M, N}$ be the weight sequence, and let $\Fb_i:=\Fb_i(\kk)$ be the corresponding weighted Fibonacci sequence. Suppose that $|\kk|=d+1$. Note that $m=2p+2\geq 4$ and $d\geq 1$.

Similarly as in the proof of Theorem \ref{theo-max-orth-typeA}, one can show that $G(0,1)$ is $\tau\omega^r$-stable if and only if $\rem{rM}_N=N-1$,  if and only if $M, N$ are coprime and $r=\Fb_d$ with $d$ odd or $r=\Fb_{d+1}-\Fb_d$ with $d$ even.  $G(0,1)$ is a maximal $r$-orthogonal subset of $\mathbb{Z}A_{m-1}$ if and only if for each vertex $(x,t)$ in the set
$$S:=\{(x,t)\mid 0\leq x\leq n, 1\leq t\leq m-1, \mbox{ or } n<x\leq m+n-2, x+t<m+n\}\backslash\{(0,1)\},$$
there is  $1\leq b\leq r$ such that $\rem{bM}_N=x+t-1$ or $\rem{(b-1)M}_N\equiv x-1 \pmod{N}$.

If $u=1$, then $n=m/2-1$. $n=1$ if and only if $m=4$. In this case $r=1$, and one can directly check that $G(0,1)$ is not maximal $1$-orthogonal. If $n=2$, then $m=6$. Then $(M, N)=2$ and thus $G(0,1)$ cannot be $\tau\omega^r$-stable. Now assume that $n>2$. Then $M=k_2n+2$ and $n=k_3\cdot 2+1$ since $s_{d+1}=1$, $d=3$ and $r=\Fb_3$. Taking  $(x,t)=(0, 2)$,
%The above condition implies that there is some $0<b\leq \Fb_3$ such that $\rem{bM}_N=1$ or $\rem{(b-1)M}_N=N-1$. However,
for each $0<b\leq \Fb_3$, we have $\rem{bM}_N\geq s_3=2$. Since $\rem{\Fb_3M}_N=N-1$, there cannot be any integer $l$ less that $\Fb_3$ such that $\rem{lM}_N=N-1$. This shows that there is no integer $1\leq b\leq r$ such that $\rem{bM}_N=0+2-1=1$ or $\rem{(b-1)M}_N\equiv 0-1\equiv N-1\pmod{N}$. Hence $G(0,1)$ is not maximal $\Fb_3$-orthogonal.

Now assume that $u>1$. Then $n=u(m-1)-m/2\geq m$, and $n=m$ if and only if $m=4$ and $u=2$. If $m=n=4$, then $(M, N)=4$ and thus $G(0,1)$ is not $\tau\omega^r$-stable.

Finally, we assume that $n>m$. Then $d\geq 3$, $s_2=n$, $s_3=m$ and $n=k_3m+s_4$. The weighted Fibonacci sequence satisfies $\Fb_0=\Fb_1=1$, $\Fb_2=2$, $\Fb_3=2k_3+1$ and $\Fb_4=k_4\Fb_3+\Fb_2$. Set
$$\mathcal{X}=\{\rem{bM}_N\mid 1\leq b\leq r\}.$$
Considering the vertices $(0,t), t=2, \cdots, m-1$, we deduce that $\{1,2,\cdots,m-2\}\subseteq \mathcal{X}$. Using the other vertices in $S$, one can deduce that $|\mathcal{X}\cap [m,m+n-1]|\geq n-(k_3+1)$.
% if $m-1\in\mathcal{X}$, then $|\mathcal{X}\cap [m,m+n-1]|\geq n-(k_3+1)$, and  $|\mathcal{X}\cap [m,m+n-1]|\geq n-k_3$ otherwise.
Note that there is a bijection between $\mathcal{X}\cap [m,m+n-1]$ and $\mathcal{X}\cap [m+n,N-1]$ sending $\rem{bM}_N$ to $\rem{(b-1)M}_N$. It follows that $|\mathcal{X}|\geq N-1-2(k_3+1)$. No matter $r=\Fb_d$ or $r=\Fb_{d+1}-\Fb_d$, we have $r<\Fb_{d+1}\leq N$. It follows that $\rem{bM}_N$, $1\leq b\leq r$ are pairwise distinct and thus $r=|\mathcal{X}|$. Hence
$$r\geq N-\Fb_3-\Fb_2\geq \Fb_{d+1}-(\Fb_3+\Fb_2).$$
Note that $s_{d+1}=1$ and $k_{d+1}=s_{d}\geq 2$. Thus $\Fb_{d+1}=k_{d+1}\Fb_d+\Fb_{d-1}>2\Fb_{d}$. If $r=\Fb_d$ with $d$ odd, then $\Fb_d>2\Fb_{d}-(\Fb_3+\Fb_2)$ and thus $\Fb_d<\Fb_4$. This forces that $d=3$. But in this case $s_4=1$ and $k_4=m$ and therefore $\Fb_3\geq \Fb_4-\Fb_3-\Fb_2=k_4\Fb_3-\Fb_3=(m-1)\Fb_3$ which is a contradiction since $m\geq 4$. If $r=\Fb_{d+1}-\Fb_d$ with $d$ even, then $\Fb_d\leq \Fb_3+\Fb_2\leq \Fb_4$. This forces that $d=4$ and $\Fb_4=\Fb_3+\Fb_2$. Hence $k_4=1$, $s_5=1$, $m=s_4+s_5$, and thus $n=k_3m+s_4=(k_3+1)m-1$. Let $a=k_3+1$. One can check that
$$r=\Fb_5-\Fb_4=2am+m-2a-3$$
and that $G(0,1)$ is indeed a maximal $r$-orthogonal subset of $\mathbb{Z}A_{m-1}$.
\end{proof}

For type $D$, there is no vertex $(0,t)$ such that $G(0,t)$ is a maximal $\rd_G(t)$-orthogonal subset. However, for type $E$, we have the following theorem.

\begin{Theo}\label{theo-max-orth-typeE}
	Let $\Lambda$ be an indecomposable representation-finite non-semisimple self-injective
	algebra of type $(E_m, u, s)$. Suppose that $X$ is an indecomposable $\Lambda$-module corresponding to the vertex
	$(0,t)$ with $\rd(X) = r$. Then $X$ is maximal $r$-orthogonal if and only if $m=7$, $t=1$, $u=9a+5$ for some non-negative integer $a$ and $r=119a+66$. In this case, $\rigdim\Lambda=119a+68$.
\end{Theo}

\begin{proof}
	 The idea is similar to the proofs of Theorem \ref{theo-max-orth-typeA} and Theorem \ref{theo-max-orth-typeA2}. For $r\leq \Fb_1$, it is easy to find a vertex $(x,y)$ such that $G(x,y)\cap H^+(\omega^i(0,t))=\emptyset$ for all $1\leq i\leq r$ and thus $G(0,t)$ is not maximal $r$-orthogonal. Now assume that $r>\Fb_1$. By Theorem \ref{theo-typeE}, the cases where $G(0,t)$ is $\tau\omega^r$-stable are as follows.

	Type $(E_7, u, 1)$:
	
	\begin{itemize}
		\item $[u]_9 = 2$, $t = 6$, $r = \Fb_1 +3\Fb_2$;
		\item $[u]_9 = 4$, $t = 6$, $r = \Fb_1 +\Fb_2$;
		\item $[u]_9 = 5$, $t = 1$, $r = \Fb_1 +3\Fb_2$;
		\item  $[u]_9 = 7$, $t = 1$, $r = \Fb_1 +\Fb_2$;
	\end{itemize}
	
	Type $(E_8, u, 1)$:
	
	\begin{itemize}
		\item $[u]_{15} = 2$, $t = 1$, $r = \Fb_3$;
		\item $[u]_{15} = 7$, $t = 1,2,7$, $r = \Fb_1 +\Fb_2$;
		\item $[u]_{15} = 11$, $t = 1,2$, $r = \Fb_1 +\Fb_2$;
	\end{itemize}

	Type $(E_6, u, s)$:
\begin{itemize}
	\item $[u]_6=1, t=1,5$, $\lfloor u/6\rfloor-s$ is even, $r=\Fb_1+4\Fb_2$.
\end{itemize}
Finally, let $M=G(0, t)$, we directly check for each $(x,y)\notin M$, whether $$G(x, y)\bigcap \left(\bigcup_{1\leq i\leq r}H^+(\omega^i(0,t))\right)\neq\emptyset.$$
The only survivor is the case that $\Lambda$ is of type $(E_7, u, 1)$, $[u]_9 = 5$, $t = 1$, $r = \Fb_1 +3\Fb_2$. In this case, suppose that $u=9a+5$, the weight sequence $\wseq{9, 17u}$ is $k_1=17a+9$, $k_2=2$ and $k_3=4$. Thus $r=\Fb_1+3\Fb_2=119a+66$. Note that $\rd_G(1)>\rd_G(y)$ for all $y\neq 1$ in this case. Therefore $\rigdim\Lambda=r+2=119a+68$.
\end{proof}

\noindent
{\bf Example.} Suppose that $\Lambda$ is the self-injective Nakayama algebra with $17$ simple modules and Loewy length $9$. By the Euclidean algorithm, we have $9=0\times 17+9$, $17=1\times 9+8$, $9=1\times 8+1, 8=8\times 1$. That is, the remainder sequence is $s_{1}=9,s_2=8, s_3=1, s_4=0$, and the weight sequence is $k_1=1, k_2=1, k_3=8$. The corresponding weighted Fibonacci sequence is $\Fb_0=1, \Fb_1=1, \Fb_2=2,\Fb_3=17$. By Theorem \ref{theo-typeA}, we get $\rd_G(1)=2(\Fb_3-\Fb_2)=30$ and $\rd_G(t)=2\Fb_2-1=3$ for all $2\leq t\leq \min\{m/2, s_2\}$. That is, $\rd_G(t)=3$ for $t=2,3,4$. By symmetry, one has $\rd_G(8)=30$ and $\rd_G(t)=3$ for $t=5,6,7$. Let $S$ be a simple $\Lambda$-module. By Theorem \ref{theo-max-orth-typeA}, $\Lambda\oplus S$ is a maximal $30$-orthogonal module and $\rigdim \Lambda=32$.

\medskip
\noindent
{\bf Acknowledgement.} The second author would like to thank China Scholarship Council for supporting her study at the University of Stuttgart and also wish to
thank the representation theory group in Stuttgart for hospitality. The research work was partially supported by NSFC (12031014).

\bigskip
{Wei Hu, School of Mathematical Sciences,  Laboratory of Mathematics and Complex Systems, MOE, Beijing Normal University, 100875 Beijing, People's Republic of  China.

 {\tt Email: huwei@bnu.edu.cn}}

 \medskip
 \bigskip
 {Xiaojuan Yin, School of Mathematical Sciences,  Laboratory of Mathematics and Complex Systems, MOE,  Beijing Normal University, 100875 Beijing, People's Republic of  China.

  {\tt Email: xiaojuan\_yin11@163.com}}
\end{document}